\documentclass[10pt,reqno]{amsart}

\usepackage[utf8x]{inputenc}
\usepackage[english]{babel}
\usepackage{amsmath,amsfonts,amssymb,amsthm,shuffle}
\usepackage[T1]{fontenc}
\usepackage{lmodern}
\usepackage{mathtools}
\usepackage{bm}

\usepackage[top=3.5cm,bottom=3.5cm,left=3.6cm,right=3.6cm]{geometry}

\usepackage[dvipsnames]{xcolor}
\usepackage[hyperindex=true,frenchlinks=true,colorlinks=true,
citecolor=Mahogany,linkcolor=DarkOrchid,urlcolor=Tan,linktocpage,
pagebackref=true]{hyperref}
\usepackage{url}

\usepackage{tikz}
\usetikzlibrary{shapes}
\usetikzlibrary{fit}
\usetikzlibrary{decorations.pathmorphing}

\usepackage{dsfont}
\usepackage{wasysym}
\usepackage{stmaryrd}
\usepackage{cite}
\usepackage{subfigure}
\usepackage{multirow}
\usepackage{enumitem}
\usepackage{multicol}
\usepackage{etex}

\numberwithin{equation}{section}
\renewcommand{\geq}{\geqslant}

\tikzstyle{Vertex}=[circle,draw=LimeGreen!80,fill=LimeGreen!8,
inner sep=1pt,minimum size=2mm,line width=1pt,font=\scriptsize]
\tikzstyle{Node}=[Vertex,draw=RoyalBlue!80,fill=RoyalBlue!8,inner sep=1.5pt]
\tikzstyle{Leaf}=[rectangle,draw=Black!70,fill=Black!16,
inner sep=0pt,minimum size=1mm,line width=1.25pt]
\tikzstyle{Edge}=[Maroon!80,cap=round,line width=1pt]
\tikzstyle{Mark1}=[draw=BrickRed!80,fill=BrickRed!8]
\tikzstyle{Mark2}=[draw=BurntOrange!80,fill=BurntOrange!8]
\tikzstyle{EdgeRew}=[->,RedOrange!80,cap=round,thick]

\newcommand{\wcz}{\xrightarrow[]{\mathcal{D}}}
\newcommand{\wc}{\xrightarrow[n\to\infty]{\mathcal{D}}}
\newcommand{\pc}{\xrightarrow[n\to\infty]{\mathrm{P}}}
\newcommand{\ip}[1]{\langle #1\rangle}

\newcommand{\iid}{\mbox{i.i.d.\frenchspacing}}

\newcommand \R{\mathbb R}

\newcommand \Z{\mathbb Z}

\newcommand \abs[1]{\left|#1\right|}

\newcommand{\EE}{\mathsf{E}\,}
\newcommand{\dd}{\mathrm{d}\,}
\newcommand{\e}{\varepsilon}

\newcommand \ens[1]{\Big( #1\Big)}

\newcommand \PP{\mathsf P}
\newcommand \TT{\mathsf T}
\newcommand{\E}[1]{\mathsf E\left(#1\right)}

\newcommand{\prt}[1]{\left(#1\right)}
\newcommand{\norm}[1]{\left\lVert #1 \right\rVert}

\newcommand{\cF}{\mathcal{F}}

\newcommand{\cD}{\mathcal{D}}

\newcommand{\cB}{\mathcal{B}}

\newcommand{\Hide}[1]{}

\numberwithin{equation}{subsection}
\renewcommand{\geq}{\geqslant}

\newtheorem{Theorem}{Theorem}[section]
\newtheorem{Proposition}[Theorem]{Proposition}
\newtheorem{Lemma}[Theorem]{Lemma}

\theoremstyle{remark}
\newtheorem{Remark}[Theorem]{Remark}
  
\title[Weak invariance principle in Besov spaces]{Weak invariance principle in Besov spaces for stationary martingale differences}

\author{Davide \textsc{Giraudo} and Alfredas \textsc{Ra\v{c}kauskas} }
 
\address{ 
Normandie Univ, UNIROUEN, CNRS,
LMRS, 76000 Rouen, France, 
Department of Mathematics and Informatics,
Vilnius  University, Naugarduko 24, LT-03225 Vilnius, Lithuania, 
}\email{davide.giraudo1@univ-rouen.fr, alfredas.rackauskas@mif.vu.lt.}
\begin{document}

\begin{abstract} 
The classical Donsker weak invariance principle is extended to a Besov spaces
framework. Polygonal line processes  build from partial sums of stationary martingale 
differences as well independent and identically distributed 
random variables are considered. The results obtained are shown to be optimal.   
\end{abstract}
\maketitle


\smallskip 


\section{Introduction  and main results}

By weak invariance principle in a topological function space, say, $E$  we understand  the weak convergence of a sequence of probability measures induced on $E $ by  normalized polygonal line processes build from partial sums of random
variables.  The choice of the space $E$ is important due to possible statistical applications
via continuous mappings. Since stronger topology generates more continuous functionals, it is beneficial to have the weak invariance principle proved in as strong as possible topological framework. 

Classical Donsker's weak invariance principle considers the space $E=C[0, 1]$ and polygonal 
line processes build from partial sums of i.i.d. centred  random variables with finite 
second moment. An intensive research has been done in order to extend Donsker's result
to a stronger topological framework as well to a larger class of random 
variables (see, e.g., \cite{MR2206313}, \cite{Rackauskas-Suquet:2004a}, \cite{Giraudo2015} and 
references therein).

In this paper we consider weak invariance principle in Besov spaces  for a class of strictly 
stationary sequence of martingale differences. To be more precise, let us first introduce 
some notation and definitions used throughout the paper.

Let $\prt{\Omega, \cF, \PP}$ be a probability space and $\TT\colon \Omega\to\Omega$ be a 
bijective bi-measurable transformation preserving the probability $\PP$. The quadruple  $(\Omega, \cF, \PP, \mathsf{T})$ is referred to as  dynamical system (see, e.g., \cite{MR833286} for some background material).  We assume that
there is a sub-$\sigma$-algebra $\cF_0\subset \cF$ such that $\mathsf{T}\cF_0\subset \cF_0$ and
by $\mathcal I$ 
we denote  the $\sigma$-algebra of the sets $A\in \mathcal F$ such that $\mathsf{T}^{-1}A=A$. 

Next we  consider 
a strictly stationary sequence $\prt{X_{j},\ j\geqslant 0}$ constructed  
as  $X_j:=f\circ \TT^j$, where $f:\Omega\to \R$  
is $\cF_0$-measurable. We define also a non-decreasing filtration $\cF_k=\TT^{-k}\cF_0$, 
$k\geqslant 1$. Note that  $\prt{X_j,  \cF_j, j\geqslant 0}$ is then a martingale differences 
sequence provided $\E{f\mid \TT\mathcal \cF_0}=0$. 

Set 
$$
S_{f,0}:=0, \ S_{f,n}:=\sum_{j=0}^{n-1}f\circ\mathsf{T}^j
, \quad n\geqslant 1.
$$ 
Our main object of 
investigation is the sequence of polygonal line processes
$\zeta_{f, n}:=(\zeta_{f,n}(t), t\in [0, 1]),$ $n\geqslant  1$,
defined by 

$$
\zeta_{f, n}(t):=S_{f,\lfloor nt\rfloor}+ (nt-\lfloor nt\rfloor)f\circ \TT^{\lfloor nt\rfloor},
$$ 

where  for a real number $a \geqslant 0$, $\lfloor a\rfloor := \max\{k\colon k \in  
\{0, 1, \ldots \}, k \le a\}$.
To define the paths space under consideration let $L_p([0, 1])$ be the space of 
Lebesgue integrable functions with exponent $p$  ($1\leqslant  p < \infty$)  and  the
norm
$$
\lVert x\rVert_p=\prt{\int_0^1\abs{x(t)}^p\dd t}^{1/p},\quad x\in L_p([0, 1]).
$$
If $x\in L_p([0, 1])$  its $L_p$-modulus of smoothness is defined as
$$
\omega_p(x, \delta) = \sup_{\abs h\leqslant \delta}\left(\int_{I_h}
\abs{x(t + h) - x(t)}^p \dd t\right)^{1/p},\quad \delta\in [0, 1],
$$
where $I_h=[0, 1]\cap[-h, 1-h]$.
Let $\alpha\in [0, 1)$. The Besov space $B^o_{p, \alpha}=B^o_{p, \alpha}([0, 1])$ is 
defined by
$$
B^o_{p, \alpha}:=\left\{x\in L_p([0, 1]): \lim_{\delta\to 0}\delta^{-\alpha}\omega_p(x, \delta)
=0\right\}.
$$
Endowed with the norm
$$
\norm{x}_{p, \alpha}=\norm{x}_p+\sup_{\delta\in (0, 1)}\delta^{-\alpha}\omega_{p}(x, \delta),
\quad x\in B^o_{p, \alpha}, 
$$
the space $B^o_{p, \alpha}$ is  a separable Banach space and the following 
embeddings are continuous:
\begin{align*}
&B^o_{p, \alpha}\hookrightarrow B^o_{p, \beta},\quad \textrm{for}\ 0\le \beta<\alpha;\\
&B^o_{p, \alpha}\hookrightarrow B^o_{q, \alpha},\quad \textrm{for}\ 1\le q<p<\infty.
\end{align*}
Each $B^o_{p, \alpha}(0, 1)$ where $p\geqslant 1, 0\leqslant \alpha <1/2,$ supports a
standard
Wiener process $W=\prt{W(t), 0\leqslant t\leqslant 1}$ (see, e.g., \cite{MR1277166}). We 
note also, that any polygonal line process  belongs to each of $B^o_{p, \alpha}$, $p\ge 1$, 
$\alpha\in [0, 1)$. 

As usually $\wcz$ denotes convergence in distribution.
\begin{Theorem}\label{thm:WIP_Besov_martingales_1}
Let $p>2$, $1/p<\alpha<1/2$ and 
\begin{equation}\label{q_alpha_p}
q(p, \alpha):=1/(1/2-\alpha+1/p).
\end{equation}
Let $(f\circ \TT^i,\  \TT^{-i}\mathcal \cF_0,\ i\geqslant 0)$ be a martingale 
differences sequence.
Assume that the following two conditions hold~:
\begin{itemize}
\item[$(i)$]

$ \lim_{t\to\infty}t^{q(p, \alpha)} \PP\left\{\abs f\geqslant t\right\}=0;
$ 
\item[$(ii)$]
$ \E{\big[\E{f^2\mid \TT\cF_0}\big]^{q(p,\alpha)/2}}<\infty.
$
\end{itemize}
Then the convergence 
$
n^{-1/2}\zeta_{f, n}\wc  \sqrt{\E{f^2\mid\mathcal I}}
\ W$ holds in the space $B_{p,\alpha}^o,$
where  $W$ is independent of $\E{f^2\mid\mathcal I}$.
\end{Theorem}

Let us note that condition (i)  is stronger for the function $f$ than its square 
integrability since (i) yields $\E{|f|^r}<\infty$ for any $r<q(p, \alpha)$ 
and $q(p, \alpha)>2$ when $\alpha >1/p$.  However as
shows our next result, condition (i) is necessary and sufficient for independent 
identically distributed (\iid) random variables. To formulate the 
result let  $Y, Y_1, Y_2, \dots$ be mean zero \iid\  random 
variables with finite variance $\sigma^2=\E{Y^2}>0$. Let $\xi_n=(\xi_n(t), 
t\in [0, 1])$, be defined by
$$
\xi_n(t)=\sum_{k=1}^{\lfloor nt\rfloor}X_k+(nt-\lfloor nt\rfloor)
X_{\lfloor nt\rfloor +1}.
$$
\begin{Theorem}\label{thm:WIP_Besov_iid_1} 
Let $p>2, 1/p<\alpha<1/2$ and let $q(p, \alpha)$ be defined by \eqref{q_alpha_p}. 
Then
\begin{equation}\label{eq:WIP_Besov_iid_1}
n^{-1/2}\sigma^{-1}\xi_n\wc W\quad \textrm{in the space}\quad B^o_{p, \alpha}
\end{equation}
if and only if 
\begin{equation}\label{eq:tail_condition_iid}
\lim_{t\to\infty}t^{q(p, \alpha)} \PP\left\{\abs Y\geqslant t\right\}=0.
\end{equation}
\end{Theorem}
Since $B_{\infty, \alpha}^o$ matches H\"older spaces we see, that Theorems  \ref{thm:WIP_Besov_iid_1} 
and  \ref{thm:WIP_Besov_martingales_1} complement the weak invariance principle obtained 
in a H\"olderian framework by \cite{Rackauskas-Suquet:2004a} and \cite{Giraudo2015}. 
Concerning condition (ii) of Theorem \ref{thm:WIP_Besov_martingales_1} we prove a need of
certain extra assumption  by a counterexample which for any dynamical system with positive entropy 
constructs a function $f$ that satisfies the condition (i) but the convergence of polygonal line processes fails. Precise result reads as follows. 

\begin{Theorem}\label{thm:counter_example_martingales}
 Let $p>2$, $1/p<\alpha<1/2$ and $q(p,\alpha)$ be given by 
 \eqref{q_alpha_p}. For each dynamical system $\prt{\Omega,\mathcal F,\PP, \TT}$ with positive 
 entropy,  there exists a function $m\colon \Omega\to \R$ and a $\sigma$-algebra 
 $\mathcal F_0$ for which $\TT\mathcal F_0\subset\mathcal F_0$ such that:
 \begin{itemize}
  \item[$(i)$] the sequence $\prt{m\circ \TT^i,\ \TT^{-i}\mathcal F_0, i\geqslant 0}$ 
  is a martingale difference sequence; 
  \item[$(ii)$] the convergence $\lim_{t\to +\infty}t^{q(p,\alpha)} \PP 
  \ens{\abs m\geqslant t}=0$ takes place;
  \item[$(iii)$] the sequence $\prt{n^{-1/2}\zeta_{m, n}}$
  is not  tight in $B^o_{p,\alpha}$.
 \end{itemize}
\end{Theorem}

As it is seen from our next results the case where $0\le \alpha\le 1/p$ is 
indeed quite different from the previously considered case  where $1/p<\alpha<1/2$.
 
\begin{Theorem}\label{thm:WIP_Besov_martingales_2}
 Let $p\geqslant 1$ and $\alpha\in [0, 1/2)\cap[0, 1/p]$.    Let $\prt{f\circ \TT^i,\ \TT^{-i}\mathcal F_0, i\geqslant 0}$  be  a martingale 
differences sequence. If  $\EE{f^2}<\infty$
then
$n^{-1/2}\zeta_{f,n}  \wc  \sqrt{\E{f^2|\mathcal{I}}} W$ in the space
$B_{p,\alpha}^o$, where  $W$ is independent of $\E{f^2\mid\mathcal I}$.
\end{Theorem}

\begin{Theorem}\label{thm:WIP_Besov_iid_2}  
Let  $\alpha$ and $p$ be as in Theorem \ref{thm:WIP_Besov_martingales_2}. 
Then
\begin{equation}\label{eq:WIP_Besov_iid_2}
n^{-1/2}\sigma^{-1}\xi_{n}\wc W\quad \textrm{in the space}\quad B^o_{p, \alpha}.
\end{equation}
\end{Theorem}
Let us note that the finiteness of the second moment $\EE Y^2$ is necessary for the convergence (\ref{eq:WIP_Besov_iid_2}).

%
%
%
%

The rest of the paper is organized as follows. In 
Section~\ref{sec:functional_analysis_prob_tools}
we shortly present needed 
information on structure of  Besov spaces and tightness of measures on these 
spaces. Section~\ref{sec:proof_alpha_>_1/p} contains proofs of 
Theorems~\ref{thm:WIP_Besov_martingales_1}, \ref{thm:WIP_Besov_iid_1} 
and \ref{thm:counter_example_martingales}  whereas Section~\ref{sec:proof_alpha_leq_1/p} 
is devoted to the 
proofs of Theorems \ref{thm:WIP_Besov_martingales_2} and \ref{thm:WIP_Besov_iid_2}. {{Finally, in Section 6 we discuss possible applications of invariance principle in the Besov framework.  }}

\section{Some functional analysis and probabilistic tools}
\label{sec:functional_analysis_prob_tools}

We denote by $\cD_j$ the set of dyadic numbers in
$[0,1]$ of level $j$, i.e.
$$
\cD_0=\{0,1\},\qquad \cD_j=
\bigl\{(2l-1)2^{-j};\;1\leqslant  l \leqslant 2^{j-1}\bigr\},\quad j\geq 1.
$$
Set
$$
\cD=\bigcup_{j\geqslant 0}\cD_j
$$
and  write for $r\in \cD_j$, 
$$
r^-:=r-2^{-j},\quad r^+:=r+2^{-j}.
$$
The triangular Faber-Schauder functions $\Lambda_r$ for $r\in \cD_j$, $j>0$, are 
$$
\Lambda_r(t)=\left\{
\begin{array}{ll}
2^j(t-r^-)    & \text{if $t\in (r^-,r]$;}\\
2^j(r^+ -t) & \text{if $t\in (r,r^+]$;}\\
0             & \text{else.}
\end{array}
\right.
$$
When $j=0$, we just take the restriction to $[0,1]$ in
the above formula, so
$$
\Lambda_0(t)=1-t,\quad \Lambda_1(t)=t,\quad t\in[0,1].
$$
\begin{Theorem}[\cite{MR1244574}] Let $p>1$ and $1/p<\alpha<1.$ The Faber-Schauder system $\{\Lambda_r, r\in D\}$ is the Schauder basis for $B^o_{p, \alpha}$:  
each $x\in B^0_{p, \alpha}$ has the unique representation,
$$
x=\sum_{r\in \cD}\lambda_r(x)\Lambda_r,
$$
where
$$
\lambda_r(x):=x(r)-\frac{x(r^+)+x(r^-)}{2},\quad r\in D_j,\;j\geq 1
$$
and in the special case $j=0$, 
$$
\lambda_0(x):=x(0),\quad \lambda_1(x):=x(1).
$$
Moreover the norm is equivalent to the sequential norm:
$$
\norm{x}_{p, \alpha}\sim \norm{x}_{p,\alpha}^{\mathrm{seq}}:=\sup_{j\geqslant 0}2^{j\alpha-i/p}
\Big(\sum_{r\in D_j}|\lambda_{r}(x)|^p\Big)^{1/p}.
$$
\end{Theorem}

The Schmidt orthogonalization procedure (with respect to inner 
product in $L_2(0, 1)$) applied 
to Faber-Schauder system leads to 
the Franklin system $\{f_k, k\geqslant 0\}$:
$$
f_k(t)=\sum_{i=0}^k c_{ik}\Lambda_i(t), \quad t\in [0, 1],
$$
with $c_{kk}>0$ for $k\geqslant 0$, where the matrix $(c_{ik})$ is uniquely determined.

\begin{Theorem}[\cite{MR1244574}] The Franklin system $\{f_n, n\geqslant 0\}$ is the
basis for $B^o_{p, \alpha}$, $p\geqslant 1, 0\leqslant \alpha<1$: each
$x\in B^o_{p, \alpha}$ has the 
unique representation,
$$
x=\sum_{k=0}^\infty x_kf_k,
$$
where $x_k=\langle x, f_k\rangle :=\int_0^1 x(t)f_k(t)\dd t$, $k\geqslant 0$.
\end{Theorem}

The following proposition is  proved in \cite{BDL:2005} for $\alpha>1/p$ but similar arguments works as well for any $0\leqslant\alpha<1$.    
\begin{Proposition}\label{Compact_sets_1} Let $p\geqslant 1$ and $0\leqslant \alpha<1$. The
set $K\subset B_{p, \alpha}^o$ is relatively compact if and only if 
\begin{itemize}
\item[$(i)$] $\sup_{x\in K}||x||_{p}<\infty,$ 
\item[$(ii)$] $\lim_{\delta\to 0}\sup_{x\in K}\delta^{-\alpha}\omega_{p, \alpha}(x, 
\delta)=0.$
\end{itemize}
\end{Proposition}
\begin{proof}
One easily checks that (i) and (ii) yields  relative compactness of $K$  in $L_p([0, 1])$. 
Therefore for any sequence $(x_n)_{n\geqslant 1}$ of $K$ there exists a subsequence, which
we denote also $(x_{n})$, converging in $L_p([0, 1])$ to some  $x\in  L_p([0, 1])$. To finish
the proof it suffices to prove that
\begin{itemize}
\item[$(a)$] $x\in B^o_{p, \alpha}$;
\item[$(b)$] $(x_n)_{n\geqslant 1}$ is a Cauchy sequence in $B^o_{p, \alpha}$.
\end{itemize}
Taking a.s. convergence subsequence  $(x_{n'})$ and applying Fatou lemma we easily obtain for any $0<\delta\leqslant 1$,
\begin{equation}\label{comp_0}
\omega_p(x, \delta)\leqslant \liminf_{n'}\omega_p(x_{n'},\delta)\leqslant \sup_n\omega_p(x_n, \delta).
\end{equation}
This yields (a). 
To prove (b) observe that for each $\e>0$ there exists $\delta_\e>0$ such that 
$\delta^{-\alpha}\sup_n \omega_p(x_n, \delta)<\e\quad \textrm{when}\quad 
\delta<\delta_\e$, hence, 
for $n,m\geqslant 1$
\begin{align*}
||x_n-x_m||_{p, \alpha}&=||x_n-x_m||_p+
\max\left\{[\sup_{0<\delta\leqslant \delta_\e}; \sup_{\delta_\e<\delta\leqslant 1}]\delta^{-\alpha}
\omega_p(x_n-x_m, \delta)\right\}\\
&\leqslant  ||x_n-x_m||_p+\e+2\delta_{\e}^{-\alpha}||x_n-x_m||_p
\end{align*} 
and we complete the proof since  $\lim_{n, m\to \infty}||x_n-x_m||_p=0$.
\end{proof}

Consider stochastic processes $Z, (Z_n)_{n\geqslant 1}$ with paths space $B^o_{p,\alpha}$ which is endowed with Borel $\sigma$-algebra $\cB(B^o_{p, \alpha})$. Let $P_Z, P_{Z_n}, n\geqslant 1,$ be the corresponding distributions. 
As generally accepted the sequence $(Z_n)$ converges in distribution to $Z$ in $B^o_{p,\alpha}$  (denoted $Z_n\wc Z$ in $B^o_{p,\alpha}$) provided $(P_{Z_n})$ converges weakly to $P_Z$:  
$\lim_{n\to\infty}\int f \dd P_{Z_n}=\int f\dd P_{Z}$ for each 
bounded continuous $f:B^o_{p, \alpha}\to \R$. The sequence $(Z_n)$ is tight in $B^o_{p, \alpha}$ if for each $\e>0$ there is a relatively compact set $K_\e\subset B^o_{p, \alpha}$ such that $\inf_{n\geqslant 1}\PP(Z_n\in K_\e)>1- \e$. 
Due to the well known Prohorov's theorem  convergence in distribution in a separable metric space  is  coherent with tightness. Indeed, to prove convergence in distribution one has to establish tightness and to ensure uniqueness of the limiting distributions. 
 
The following tightness criterion is obtained from Proposition  \ref{Compact_sets_1}.
\begin{Theorem}\label{thm:necessary_and_sufficient_condition_for_tightness}
The sequence $\prt{Z_n}$ of random processes with paths 
in  $B^o_{p,\alpha}(0, 1)$ is tight if and only if the following two conditions are satisfied: 
\begin{itemize} 
\item[(i)] $\lim_{b\to\infty}\sup_{n\geqslant 1} \PP\ens{||Z_n||_{p}>b}=0$;  
\item[(ii)]  for each $\e>0$ 
$$
\lim_{\delta\to 0}\sup_{n\geqslant 1} \PP\ens{\omega_{p, \alpha}(Z_n, \delta)\geqslant \e}=0.
$$
\end{itemize}
\end{Theorem}

\begin{proof} See, e.g., the proof of Theorem~8.2. in 
\cite{Billingsley:1968}. 
\end{proof}


\begin{Theorem}\label{thm:NSC_tightness} Let $1>\alpha>1/p$.
The sequence $(Z_n)$  of random elements in
the Besov space $B^o_{p, \alpha}(0, 1)$ is tight
if and only if the following two conditions are satisfied:
\begin{itemize}
\item [(i)]\  { $\displaystyle \lim_{b\to \infty}
\sup_n \PP\bigl\{||Z_n||_p>b\bigr\}= 0$;}

\item[(ii)]\ \ {  for each $\e>0$,
$\displaystyle
\lim_{J\to \infty}\limsup_{n\to\infty} 
\PP\Big(\sup_{j\geqslant J}2^{j\alpha-j/p }\Big(\sum_{r\in D_j}
\abs{\lambda_r(\zeta_n)}^p\Big)^{1/p}>\e\Big)= 0.
$}
\end{itemize}
\end{Theorem} 
\begin{proof} It is just a corollary of tightness criterion established in \cite{Suquet:1999} for Schauder decomposable Banach spaces as Besov spaces $B^o_{p, \alpha}$ with $\alpha>1/p,$ are such. \end{proof} 

\section{Proofs: the case $\alpha>1/p$}
\label{sec:proof_alpha_>_1/p}

We start this section with some auxiliary results which could be helpful when dealing 
with weak invariance principle for stationary sequences.
Throughout we denote
$$
W_{f,n}=n^{-1/2}\zeta_n.
$$
 
\subsection{Auxiliary results}

\begin{Lemma}\label{remark_fdd_1}  Let $p\geqslant 1$ and $\alpha>1/p$. Assume that $Z$ is a random element in
boths spaces $C[0, 1]$ and $B^o_{p, \alpha}$. Then for any 
stationary sequence $\prt{f\circ \TT^j}$ if 
\begin{itemize}
\item[$(i)$] $W_{f,n}\wc Z$ in $C[0, 1]$, and 
\item[$(ii)$] $(W_{f,n})$ is tight in $B^o_{p, \alpha}$, 
\end{itemize}
then $W_{f,n}\wc Z$ in $B^o_{p, \alpha}$. 
\end{Lemma}
\begin{proof} From (ii) we have that each subsequence of $(W_{f,n})$ has further 
subsequence that converges in distribution. If $W_{f,n'}\wc Y'$ and $W_{f, n''}\wc Y''$ 
then  we have that $\lambda_r(W_{f,n'})\wc \lambda_r(Y')$ and
$\lambda_r(W_{f,n''})\wc \lambda_r(Y'')$ for any dyadic number $r$. But (i) gives 
that both $\lambda_r(Y')$ and $\lambda_r(Y'')$ have the same distribution as $\lambda_r(Z)$.
Since Schauder coefficients $(\lambda_r(Z))$ determines the distribution of $Z$  we can 
conclude that $Y'$ and $Y''$ are 
equally distributed with $Z$. This ends the proof.
\end{proof} 

For polygonal line processes build from any stationary sequence  the tightness
conditions given in Theorem \ref{thm:NSC_tightness} can be simplified. 

\begin{Theorem}\label{cor:tightness_stationary_sequences}
  Let $p\geqslant 1$ and $\alpha>1/p$. 
  The sequence $(W_{f,n})$ is tight in $B_{p, \alpha}^o$
  provided that 
 \begin{equation}\label{eq:tightness_stationary_sequences}
  \lim_{J\to +\infty}
  \limsup_{N\to +\infty}
 \sum_{j=J}^N2^{j}\int_0^1x^{p-1}\PP\Big(2^{-(N-j)/2}
 \max_{1\leqslant k\leqslant 2^{N-j}}\left|S_{f,k}\right|>x
2^{{j}/{q(p,\alpha)}}\Big)\mathrm dx=0.
 \end{equation}
\end{Theorem}
\begin{proof}
 Assume that $f$ satisfies \eqref{eq:tightness_stationary_sequences}.
 We have to show that $(W_{f,n})$ satisfies the conditions of 
 Theorem~\ref{thm:NSC_tightness}. 
First we check its condition (i). {Since 
$$
||W_{f,n}||_p\le \sup_{0\le t\le 1}|W_{f, n}(t)|=n^{-1/2}\max_{0\leqslant t\leqslant 1}|S_{f,k}|,
$$
the proof of (i) reduces to 
\begin{equation}\label{red_1}
  \lim_{b\to +\infty}\sup_{N\geqslant 1}\PP\Big(2^{-N/2}
  \max_{1\leqslant k\leqslant 2^N}\abs{S_{f,k}} >b \Big)=0.
 \end{equation}
 Notice that \eqref{eq:tightness_stationary_sequences} implies 
 (by considering the term of index $J$ in the sum) that 
 $$
  \lim_{J\to +\infty}\limsup_{N\to +\infty}2^J
  \PP\left(2^{-(N-J)/2}
 \max_{1\leqslant k\leqslant 2^{N-J}}\left|S_{f,k}\right|>\frac{1}{2}
2^{{J}/{q(p,\alpha)}}\right)=0,
$$ 
and consequently 
$$
\lim_{J\to +\infty} \limsup_{N\to +\infty}2^J
  \PP\left(2^{-N/2}
 \max_{1\leqslant k\leqslant 2^{N}}\left|S_{f,k}\right|>
2^{{J}/{q(p,\alpha)}}\right)=0.
$$ 
For a fixed $\varepsilon$, we choose $J_0$ such that 
$$ 
 \limsup_{N\to +\infty}2^{J_0}
  \PP\left(2^{-N/2}
 \max_{1\leqslant k\leqslant 2^{N}}\left|S_{f,k}\right|>
2^{{J_0}/{q(p,\alpha)}}\right)< 2\varepsilon. 
$$ 
There exists an integer $N_0$ such that for $N\geqslant N_0$, 
$$
\PP\left(2^{-N/2}
 \max_{1\leqslant k\leqslant 2^{N}}\left|S_{f,k}\right|>
2^{{J_0}/{q(p,\alpha)}}\right)< \varepsilon. 
$$
Since $\sup_{N\leqslant N_0}\PP\left(2^{-N/2}
 \max_{1\leqslant k\leqslant 2^{N}}\left|S_{f,k}\right|>
b\right)\to 0$ as $b$ goes to infinity, we can choose $b'_0$ such that 
$\max_{N\leqslant N_0}\PP\left(2^{-N/2}
 \max_{1\leqslant k\leqslant 2^{N}}\left|S_{f,k}\right|>
b'_0\right)<\varepsilon$. Taking $b_0:=\max\prt{2^{{J_0}/{q(p,\alpha)}}/2, 
b'_0}$, we have for $b\geqslant b_0$, 
$$
 \sup_{N\geqslant 1}\PP\left(2^{-N/2}
 \max_{1\leqslant k\leqslant 2^{N}}\left|S_{f,k}\right|>
b\right)<\varepsilon,
$$ 
which proves (\ref{red_1}) and the same time  (i) of Theorem~\ref{thm:NSC_tightness}.
}

Now, let us  prove condition (ii) of Theorem~\ref{thm:NSC_tightness}.  
 Since 
 \begin{align*}
  \sum_{j=J}^N2^{j}&\int_0^1x^{p-1}\PP\left(2^{-(N-j)/2}
 \max_{1\leqslant k\leqslant 2^{N-j}}\left|S_{f,k}\right|>x
2^{{j}/{q(p,\alpha)}}\right)\mathrm dx
\\
 \geqslant &2^N\int_{1/2}^{3/4}x^{p-1}\PP\left(
\left|f\right|>x
2^{{N}/{q(p,\alpha)}}\right)\mathrm dx \\
\geqslant&{ 2^N\prt{1/2}^{p-1}\PP\left(
\left|f\right|>\frac 34
2^{{N}/{q(p,\alpha)}}\right),}
 \end{align*}
 we infer that condition \eqref{eq:tightness_stationary_sequences} implies 
 \begin{equation}\label{eq:tail_condition}
  \lim_{t \to + \infty}t^{q(p,\alpha)}
\PP\ens{\abs f>t}=0.
  \end{equation}
We first prove that for each positive $\varepsilon$, 
$$
 \limsup_{n\to\infty} 
\PP\Big(\sup_{j\geqslant \lfloor \log n\rfloor+1}2^{j\alpha-j/p }\Big(\sum_{r\in D_j}
\abs{\lambda_r(W_{f,n})}^p\Big)^{1/p}>\e\Big)= 0.
$$ 
We shall actually prove that 
\begin{equation}\label{eq:remove_j_geq_log_n}
\limsup_{n\to\infty} 
\PP\Big(\sup_{j\geqslant \lfloor \log n\rfloor+1}2^{j\alpha-j/p }\Big(\sum_{r\in D_j}
\abs{W_{f,n}(r^+)-W_{f,n}(r) }^p\Big)^{1/p}>\e\Big)= 0,
\end{equation}
since the differences $W_{f,n}(r)-W_{f,n}(r^-)$ can be treated similarly. 
To this aim, define for a fixed $j\geqslant \lfloor \log n\rfloor+1$ the sets 
$$ 
 I_k:=\ens{r\in D_j, \ \frac kn\leqslant r<r^+<\frac{k+1}n},\quad 0\leqslant k\leqslant n-1;
$$ 
and
$$ J_k:=\ens{r\in D_j, \frac kn\leqslant r <\frac{k+1}n\leqslant r^+
 < \frac{k+2}n},\quad 0\leqslant k\leqslant n-2.
$$ 
Assume that $r$ belongs to $I_k$. Then $\lfloor nr\rfloor=
\lfloor nr^+\rfloor=k$. We thus have 
\begin{equation}\label{intermediate_step_proof_tightness_criterion1}
 \abs{W_{f,n}(r^+)-W_{f,n}(r) }=\abs{\prt{nr^+-nr}f\circ \TT^k/\sqrt n   } 
 =n^{1/2}2^{-j}\abs{f\circ \TT^k}.
\end{equation}
Now, assume that $r$ belongs to $J_k$. Then 
\begin{align*}
 \abs{W_{f,n}(r^+)-W_{f,n}(r) }&\leqslant 
 \abs{W_{f,n}(r^+)-W_{f,n}((k+1)/n)   }+ \abs{W_{f,n}((k+1)/n)-W_{f,n}(r)}\\
 &=n^{-1/2}\prt{\abs{\prt{nr^+-(k+1)} f\circ \TT^{k+1}}
 +\abs{ f\circ \TT^k-\prt{nr-k}f\circ \TT^k }},
\end{align*}
and using the fact that $0\leqslant nr^+-(k+1)\leqslant nr^+-nr=n2^j$ and 
$0\leqslant 1-\prt{nr-k} \leqslant\prt{nr^+-k}-\prt{nr-k}=n2^{-j} $, we get 
\begin{equation}\label{intermediate_step_proof_tightness_criterion2}
 \abs{W_{f,n}(r^+)-W_{f,n}(r) } \leqslant \sqrt n2^{-j}
 \prt{\abs{f\circ \TT^k}+\abs{f\circ \TT^{k+1} }}.
\end{equation}
Since $D_j=\{1-2^{-j}\}\cup\bigcup_{k=0}^{n-1}I_k\cup
\bigcup_{k=0}^{n-2}J_k$ and for $r=1-2^{-j}$, 
$$ \abs{W_{f,n}(r^+)-W_{f,n}(r) } 
\leqslant n^{-1/2}2^{-j}\abs{f\circ \TT^n},
$$
we have in view of 
\eqref{intermediate_step_proof_tightness_criterion1} and 
\eqref{intermediate_step_proof_tightness_criterion2}, 
\begin{align*}
 \Big(\sum_{r\in D_j}
\abs{W_{f,n}(r^+)-W_{f,n}(r) }^p\Big)^{1/p}
&\leqslant n^{-1/2}2^{-j}\abs{f\circ \TT^n} 
+n^{1/2}2^{-j}\Big(\sum_{k=0}^{n-1} \operatorname{Card}\prt{I_k} \abs{f\circ \TT^k}^p
\Big)^{1/p}\\
&+n^{1/2}2^{-j}\Big(\sum_{k=0}^{n-2} \operatorname{Card}\prt{J_k}
\prt{ \abs{f\circ \TT^k}+\abs{f\circ \TT^{k+1} } }^p
\Big)^{1/p}.
\end{align*}
We now have to bound $\operatorname{Card}\prt{I_k}$ and $\operatorname{Card}\prt{J_k}$. 
Let $1\leqslant l\leqslant 2^{j}$. If $(2l-1)2^{-j}$ belongs to $I_k$, then 
we should have $2^jk/n\leqslant 2l-1<2l<2^j(k+1)/n$ hence 
$2^{j-1}k/n\leqslant l<2^{j-1}(k+1)/n$ and it follows that $I_k$ cannot have 
more than $ 2^j/n$ elements. If $(2l-1)2^{-j}$ belongs to $J_k$, then 
we should have $2^j(k+1)/n\leqslant 2l<2^j(k+2)/n$ and we deduce that the 
cardinal of $J_k$ does not exceed $2^j/n$. Therefore, we have 
\begin{equation*}
 \Big(\sum_{r\in D_j}
\abs{W_{f,n}(r^+)-W_{f,n}(r) }^p\Big)^{1/p}
\leqslant 3n^{1/2}2^{-j}\prt{2^j/n}^{1/p}\Big(\sum_{k=0}^{n}  \abs{f\circ \TT^k}^p
\Big)^{1/p}.
\end{equation*}
and 
\begin{multline*}
 \sup_{j\geqslant \lfloor \log n\rfloor+1}2^{j\alpha-j/p }\Big(\sum_{r\in D_j}
\abs{\lambda_r(W_{f,n})}^p\Big)^{1/p} \leqslant  
 3\sup_{j\geqslant \lfloor \log n\rfloor+1}n^{1/2}2^{-j}2^{j\alpha-j/p}
 \Big(\sum_{k=0}^{n}  \abs{f\circ \TT^k}^p
\Big)^{1/p}\\
\leqslant n^{-1/2+\alpha-1/p}
 \Big(\sum_{k=0}^{n}  \abs{f\circ \TT^k}^p
\Big)^{1/p}=n^{-1/q(p,\alpha)}
 \Big(\sum_{k=0}^{n}  \abs{f\circ \TT^k}^p\Big)^{1/p}.
\end{multline*}
We thus have to prove that the latter term goes to zero in probability 
as $n$ goes to infinity. 


\begin{Lemma}\label{lem:LLN}
 Let $f$ be a function such that \eqref{eq:tail_condition} holds. Then 
$$n^{-1/q(p,\alpha)}
 \Big(\sum_{k=0}^{n}  \abs{f\circ \TT^k}^p
\Big)^{1/p}\pc 0.
$$
\end{Lemma}
\begin{proof}
For fixed $\delta$ and $n$, define $f':=f\mathbf 1\big(\abs f\leqslant \delta 
n^{1/q(p,\alpha)}\big)$ and $f''=f-f'$. By Markov's inequality, we have with $q=q(p, \alpha)$
\begin{align}\label{intermediate_step_proof_tightness_criterion3a}
 \PP\Big(n^{-1/q}
 \Big(\sum_{k=0}^{n}  \abs{f'\circ \TT^k}^p
\Big)^{1/p}>\e \Big)&\leqslant 
\e^{-p} n^{-p/q} \sum_{k=0}^{n}  \EE{\abs{f'\circ \TT^k}^p}\nonumber\\
& \leqslant 2\e^{-p} n^{1-p/q} \EE{\abs{f'}^p}.
\end{align}
Now, note that 
\begin{align*}
 \EE{\abs{f'}^p}&=p\int_0^{ \delta n^{1/q} }
 t^{p-1}\PP(\abs{f'}>t)\mathrm dt
\leqslant p\int_0^{ \delta n^{1/q} }
 t^{p-q-1} \mathrm dt \cdot 
 \sup_{s>0}s^{q} \PP\ens{\abs f>s}\\
 &=\frac{p}{p-q}\delta^{(p-q)/{q}}    
 n^{{(p-q)}/{q}}\cdot 
 \sup_{s>0}s^{q} \PP\ens{\abs f>s},
\end{align*}
hence by \eqref{intermediate_step_proof_tightness_criterion3a}, we have 
\begin{equation}\label{intermediate_step_proof_tightness_criterion4}
 \PP\Big(n^{-1/q}
 \Big(\sum_{k=0}^{n}  \abs{f'\circ \TT^k}^p
\Big)^{1/p}>\e \Big)\leqslant \e^{-p}\frac{2p}{p-q}\delta^{{(p-q)}/{q}}.
\end{equation}
Notice also that 
\begin{equation}\label{intermediate_step_proof_tightness_criterion5}
 \PP\Big(n^{-1/q}
 \Big(\sum_{k=0}^{n}  \abs{f''\circ \TT^k}^p
\Big)^{1/p}>\e \Big)\leqslant (n+1)\PP(\abs f> \delta n^{1/q}). 
\end{equation}
The combination of \eqref{intermediate_step_proof_tightness_criterion4} and 
\eqref{intermediate_step_proof_tightness_criterion5} gives 
$$ \limsup_{n\to + \infty}\PP\Big(n^{-1/q}
 \Big(\sum_{k=0}^{n}  \abs{f\circ \TT^k}^p
\Big)^{1/p}>\e \Big)\leqslant\e^{-p}\frac{2p}{p-q}
\delta^{{(p-q)}/{q}} 
$$ 
and since $\delta$ is arbitrary and $p>q$, this concludes the proof of Lemma~\ref{lem:LLN}.
\end{proof}

An application of the Lemma~\ref{lem:LLN} gives \eqref{eq:remove_j_geq_log_n}. 
Now, we have to prove that 
$$ 
\lim_{J\to \infty}\limsup_{n\to\infty} 
\PP\Big(\sup_{J \leqslant j\leqslant \lfloor \log n\rfloor}
2^{j\alpha-j/p }\Big(\sum_{r\in D_j}
\abs{W_{f,n}(r^+)-W_{f,n}(r)}^p\Big)^{1/p}>\e\Big)= 0.
$$ 
It suffices to prove that 
\begin{equation}\label{reduction_1}
 \lim_{J\to \infty}\limsup_{n\to\infty} 
\PP\Big(n^{-1/2}\sup_{J \leqslant j\leqslant \lfloor \log n\rfloor}
2^{j\alpha-j/p }\Big(\sum_{r\in D_j}
\abs{S_{\lfloor nr^+\rfloor}-S_{\lfloor nr\rfloor}}^p\Big)^{1/p}>\e \Big)= 0.
\end{equation}
Indeed, we have 
$$ 
\abs{W_{f,n}(r^+)-W_{f,n}(r)} \leqslant n^{-1/2}
\prt{\abs{S_{f,\lfloor nr^+\rfloor}-S_{f,\lfloor nr\rfloor}}
 +\abs{f\circ \TT^{\lfloor nr^+\rfloor }}+\abs{f\circ \TT^{\lfloor nr\rfloor }}},
$$ 
and for $j\leqslant \lfloor \log n\rfloor$, $2^j\leqslant n$, so that 
the set $\{\lfloor nr^+\rfloor , \lfloor nr\rfloor , r\in D_j\}$ consists of distinct elements. Therefore, 
{{$$
\sup_{1 \leqslant j\leqslant \lfloor \log n\rfloor}2^{j\alpha-j/p }\Big(\sum_{r\in D_j}
\Big[n^{-1/2}
\abs{f\circ \TT^{\lfloor nr^+\rfloor }}+\abs{f\circ \TT^{\lfloor nr\rfloor }}\Big]^p\Big)^{1/p}
>\e\Big)\leqslant 
2n^{-1/q} \Big(\sum_{k=0}^{n}\abs{f\circ \TT^k}  \Big)^{1/p},
$$
}}
and this quantity goes to zero in probability by Lemma~\ref{lem:LLN}.
The proof of (\ref{reduction_1}) reduces to establish that for each positive $\varepsilon$, 
 \begin{equation}\label{eq:sufficient_condiiton_proof_corollary}
  \lim_{J\to +\infty}\lim_{n\to +\infty}
  \sum_{j=J}^{\lfloor\log n\rfloor }\PP\left(A_{n,j}\right)=0,
 \end{equation}
where 
$$
 A_{n,j}:=\left\{\sum_{l=1}^{2^{j-1}}\left|S_{\lfloor n2l2^{-j}\rfloor }(f)
 -S_{\lfloor n(2l-1)2^{-j}\rfloor }(f)\right|^p>\varepsilon n^{p/2}2^{j(1-p\alpha)}\right\}.
$$ 
We now bound $\PP\left(A_{n,j}\right)$ by splitting the probability over the set 
$$
 B_{n,j}:=\bigcup_{l=1}^{2^{j-1}}\left\{\left|S_{f, \lfloor n2l2^{-j}\rfloor }
 -S_{f, \lfloor n(2l-1)2^{-j}\rfloor }\right|>n^{1/2}2^{j(1/p-\alpha)}\right\}.
$$ 
One bounds $\PP\left(A_{n,j}\cap B_{n,j}\right)$ by $\PP(B_{n,j})$, which can in turn be bounded by 
$$
 \sum_{l=1}^{2^{j-1}}\PP\left(\left|S_{f,\lfloor n2l2^{-j}\rfloor }
 -S_{f,\lfloor n(2l-1)2^{-j} \rfloor}\right|>n^{1/2}2^{j(1/p-\alpha)}\right)
$$ 
and 
thanks to stationarity and the fact that
\begin{equation}\label{eq:bound_of_dyadics}
 \lfloor n2l2^{-j}\rfloor -
\lfloor n(2l-1)2^{-j}\rfloor \leqslant n2^{-j}+1\leqslant 2n2^{-j},
\end{equation}
 we obtain 
\begin{align}
\PP\left(A_{n,j}\cap B_{n,j}\right)&\leqslant 
2^{j-1}\PP\left(\max_{1\leqslant k\leqslant \lfloor 2n2^{-j}\rfloor }
\left|S_{f,k}\right|>n^{1/2}2^{j(1/p-\alpha)}\right)\nonumber\\
&\leqslant 2^{j-1}2^{p-1}\int_{1/2}^1t^{p-1}\PP\left(\max_{1\leqslant k\leqslant \lfloor 2n2^{-j}\rfloor }
\left|S_{f,k}\right|>tn^{1/2}2^{j(1/p-\alpha)}\right)\mathrm{d}t.\label{eq:bound_Anj_Bnj} 
\end{align}
Now, in order to bound $\PP\left(A_{n,j}\cap B_{n,j}^c\right)$, we start by the pointwise inequalities 
\begin{align*}
\varepsilon n^{p/2}&2^{j(1-p\alpha)}\mathbf 1(A_{n,j}\cap B_{n,j}^c)
\leqslant \sum_{l=1}^{2^{j-1}}\left|S_{f,\lfloor n2l2^{-j}\rfloor }
 -S_{f,\lfloor n(2l-1)2^{-j}\rfloor }\right|^p\mathbf 1(A_{n,j}\cap B_{n,j}^c)\\
 &\leqslant \sum_{l=1}^{2^{j-1}}\left|S_{f,\lfloor n2l2^{-j}\rfloor }
 -S_{f,\lfloor n(2l-1)2^{-j}\rfloor }\right|^p 
\mathbf 1\left\{
 \left|S_{f,\lfloor n2l2^{-j}\rfloor }
 -S_{f,\lfloor n(2l-1)2^{-j}\rfloor }\right|\leqslant n^{1/2}2^{j(1/p-\alpha)}\right\}.
\end{align*}
Integrating and using the fact that for a non-negative random variable $Y$ and a positive 
$R$, 
$$ \E{Y^p\mathbf 1\left\{Y\leqslant R\right\} }
 =pR^p\int_0^1t^{p-1}\PP\left(Y>Rt\right)\mathrm dt,
$$ 
we derive by stationarity and \eqref{eq:bound_of_dyadics} that 
\begin{equation}\label{eq:bound_Anj_Bnj_complement}
 \PP\left(A_{n,j}\cap B_{n,j}^c\right)\leqslant \frac{p}{\varepsilon}2^{j-1}
 \int_0^1t^{p-1}
 \PP\left(\max_{1\leqslant k\leqslant \lfloor 2n2^{-j}\rfloor }
\left|S_{f,k}\right|>tn^{1/2}2^{j(1/p-\alpha)}\right)\mathrm dt.
\end{equation}
Let us denote by $K$ a constant depending only on $p$ and $\varepsilon$
which may change from line to line. 
By \eqref{eq:bound_Anj_Bnj} and \eqref{eq:bound_Anj_Bnj_complement}, we derive that 
\begin{equation*}
 \sum_{j=J}^{\lfloor \log n\rfloor }\PP\left(A_{n,j}\right)
 \leqslant K\sum_{j=J}^{\lfloor \log n\rfloor }2^{j}
 \int_{0}^1t^{p-1}\PP\left(\max_{1\leqslant k\leqslant \lfloor n2^{1-j}\rfloor }
\left|S_{f,k}\right|>tn^{1/2}2^{j(1/p-\alpha)}\right)\mathrm dt.
\end{equation*}
If $2^{N}\leqslant n<2^{N+1}$, then we have 
\begin{equation*}
 \sum_{j=J}^{\lfloor \log n\rfloor }\PP\left(A_{n,j}\right)
 \leqslant K\sum_{j=J}^{N}2^{j}
 \int_{0}^1t^{p-1}\PP\left(\max_{1\leqslant k\leqslant 2^{N+2-j}}
\left|S_{f,k}\right|>t2^{N/2}2^{j(1/p-\alpha)}\right)\mathrm dt,
\end{equation*}
hence 
\begin{equation*}
 \sum_{j=J}^{\lfloor \log n\rfloor }\PP\left(A_{n,j}\right)
 \leqslant K\sum_{j=J}^{N+2}2^{j}
 \int_{0}^1s^{p-1}\PP\left(\max_{1\leqslant k\leqslant 2^{N+2-j}}
\left|S_{f,k}\right|>s2^{\frac{N+2}2}2^{j(1/p-\alpha)}\right)\mathrm ds.
\end{equation*}
Splitting the integral into two parts, we infer that 
\begin{multline}\label{eq:intermediate_step_proof_tightness_criterion}
 \limsup_{n\to +\infty}\sum_{j=J}^{\lfloor \log n\rfloor }\PP\left(A_{n,j}\right)\\
 \leqslant K\limsup_{N\to +\infty}
 \sum_{j=J}^{N+2}2^{j}
 \int_{0}^1s^{p-1}\PP\prt{\max_{1\leqslant k\leqslant 2^{N+2-j}}
\left|S_{f,k}\right|>s2^{{(N+2)}/2}2^{j(1/p-\alpha)}}\mathrm ds
\end{multline}
and the limit of the latter quantity as $J$ goes to infinity is zero by 
\eqref{eq:tightness_stationary_sequences}. This concludes the proof of 
Theorem~\ref{cor:tightness_stationary_sequences}.
\end{proof}

\begin{Remark}
Using deviation inequalities, similar results as those found for the 
H\"olderian weak invariance principle for stationary mixing and 
$\tau$-dependent sequences in \cite{Giraudo2015} can be found 
for Besov spaces. 
\end{Remark}

%

{\begin{Lemma}[Proposition~3.5 in \cite{giraudo:hal-01280885} ]\label{lemma_maximal_inequality} For any $q>2$, there
  exists a constant $c(q)$ such that if $
  \prt{f\circ \TT^i}_{i\geqslant 0}$ is a martingale differences sequence 
  with respect to the filtration $\prt{\TT^{-i}\mathcal F_0}_{i\geqslant 0}$ then for each 
  integer $n\geqslant 1$, 
\begin{multline}\label{eq:tail_stationary_martingale_quadratic_variance}
  \PP\ens{\frac 1{\sqrt n}\max_{1\leqslant i\leqslant n}\abs{S_{f,i}} \geqslant t}
  \leqslant c(q)n
 \int_0^{1}\PP\ens{\abs f\geqslant \sqrt n ut}u^{q-1}\mathrm du+\\
 +c(q)\int_0^\infty
 \PP\ens{
 \left(\E{f^2\mid \TT\mathcal F_0}\right)^{1/2}>vt}
 \min\ens{v,v^{q-1}}\mathrm dv.
 \end{multline}
    \end{Lemma}
}


\subsection{Proof of Theorem  \ref{thm:WIP_Besov_martingales_1}}
 
Acording to Lemma \ref{remark_fdd_1} we need only to  prove 
that the sequence $\left(n^{-1/2}\zeta_{f,n}\right)_{n\geqslant 1}$ is 
tight in $B_{p,\alpha}^o$. 
To this aim, we have to check the condition \eqref{eq:tightness_stationary_sequences}
of Theorem~\ref{cor:tightness_stationary_sequences}. 
For  fixed $N$ and $J$ such that $N\geqslant J$, $j\in \ens{J,\dots,N}$ and 
 $x\in [0,1]$, we have by \eqref{eq:tail_stationary_martingale_quadratic_variance}
 of Lemma \ref{lemma_maximal_inequality}, 
\begin{multline}
 \PP\Big(2^{-\frac{N-j}2}
 \max_{1\leqslant k\leqslant 2^{N-j}}\left|S_{f,k}\right|>x
 2^{\frac{j}{q(p,\alpha)}}\Big)\\
 \leqslant c(q)2^{N-j}\int_0^1\PP\ens{\abs f\geqslant 2^{\frac{N-j}2}
 x 2^{\frac{j}{q(p,\alpha)}}u}u^{q-1}\mathrm du\\
 +c(q)\int_0^\infty
 \PP\ens{
 \left(\E{f^2\mid \TT\mathcal F_0}\right)^{1/2}>vx
 2^{\frac{j}{q(p,\alpha)}}}
 \min\ens{v,v^{q-1}}\mathrm dv,
\end{multline}
from which we infer that 
\begin{multline}\label{eq:checking_tightness_criterion_martingales}
 \sum_{j=J}^N2^{j}\int_0^1x^{p-1}\PP\Big(2^{-\frac{N-j}2}
 \max_{1\leqslant k\leqslant 2^{N-j}}\left|S_{f,k}\right|>x
2^{\frac{j}{q(p,\alpha)}}\Big)\mathrm dx\\
\leqslant  c(q)2^N\sum_{j=J}^N\int_0^1x^{p-1}\int_0^1\PP\ens{\abs f\geqslant 2^{\frac{N}2}
 x 2^{j\left(1/p-\alpha\right)}u}\mathrm dxu^{q-1}\mathrm du\\
 +c(q)\int_0^\infty\int_0^1x^{p-1}\sum_{j=J}^N2^j
 \PP\ens{
 \left(\E{f^2\mid \TT\mathcal F_0}\right)^{1/2}>vx
 2^{\frac{j}{q(p,\alpha)}}}
 \min\ens{v,v^{q-1}}\mathrm dv\mathrm dx\\
 =:A(N,J)+B(N,J).
\end{multline}
Using the fact that 
\begin{equation*}
 \PP\ens{\abs f\geqslant t}\leqslant t^{-q(p,\alpha)}
 \sup_{s\geqslant t}s^{q(p,\alpha)}\PP\ens{\abs f\geqslant s}, 
\end{equation*}
we derive the bound 
\begin{multline*}
 A(N,J)\leqslant 
 c(q)2^{N\left(1-q(p,\alpha)/2\right)}\sum_{j=0}^N2^{j\left(1/p-\alpha\right)}
 \int_0^1\int_0^1x^{p-q(p,\alpha)-1}u^{q-q(p,\alpha)-1}\\
 \sup\ens{s^{q(p,\alpha)}\PP\ens{\abs f\geqslant s},
 s\geqslant 2^{\frac{N}2}
 xu 2^{j\left(1/p-\alpha\right)}}\mathrm dx\mathrm du.
\end{multline*}
Since $j\leqslant N$, we have $2^{\frac{N}2}
 xu 2^{j\left(1/p-\alpha\right)}\geqslant xu2^{N/q(p,\alpha)}$ and 
 accounting the inequality
 $\sum_{j=0}^N2^{j\left(1/p-\alpha\right)}\leqslant 2^{N\left(1/p-\alpha\right)}/
 (1-2^{1/p-\alpha})$, we obtain 
\begin{equation*}
 A(N,J)\leqslant 
 c(q)
 \int_0^1\int_0^1x^{p-q(p,\alpha)-1}u^{q-q(p,\alpha)-1}
 \sup\ens{s^{q(p,\alpha)}\PP\ens{\abs f\geqslant s},
 s\geqslant 2^{\frac{N}{q(p,\alpha)}}xu }\mathrm dx\mathrm du.
\end{equation*} 
Since $p>q(p,\alpha)$ and $q>q(p,\alpha)$, the integral 
$\int_0^1\int_0^1u^{q-q(p,\alpha)-1}x^{p-q(p,\alpha)-1}\mathrm dx\mathrm du$ is convergent and 
we infer by the monotone convergence theorem that 
\begin{equation}\label{eq:limit_of_ANJ}
 \forall J\geqslant 1,\quad \lim_{N\to +\infty} A(N,J)=0.
\end{equation}
Now, in order to control $B(N,J)$, we use the following elementary inequality: if $Y$ is 
a non-negative random variable, then for each $J\geqslant 1$, 
\begin{equation*}
\sum_{j\geqslant J}2^j\PP\ens{Y\geqslant 
 2^{j/q(p,\alpha)}}\leqslant 2\E{Y^{q(p,\alpha)}
 \mathbf 1\ens{Y\geqslant 2^{J/q(p,\alpha)}} }.
\end{equation*}
Applying this to $Y:=\left(\E{f^2\mid \TT\mathcal F_0}\right)^{1/2}/(vx)$, 
we obtain that 
\begin{multline*}
 B(N,J)\leqslant c(q)\int_0^\infty\int_0^1x^{p-q(\alpha)-1}
 \E{ 
 \left[\E{f^2\mid \TT\mathcal F_0}\right]^{q/2} 
 \mathbf 1\ens{\left(\E{f^2\mid \TT\mathcal F_0}\right)^{1/2}
 \geqslant vx2^{J/q(p,\alpha)}}}\\
 \cdot\min\ens{v,v^{q-1}}v^{-q(p,\alpha)}\mathrm dv\mathrm dx.
\end{multline*}
Here again, we conclude by monotone convergence that 
\begin{equation}\label{eq:limit_of_BNJ}
\lim_{J\to +\infty} \sup_{N\geqslant 1}B(N,J)=0,
\end{equation}
since the integrals $\int_0^1x^{p-q(\alpha)-1}\mathrm dx$ and 
$\int_0^{+\infty}\min\ens{v,v^{q-1}}v^{-q(p,\alpha)}\mathrm dx$ are finite (as 
$q>q(\alpha)$).

Tightness of the sequence $\left(W_{f,n}\right)_{n\geqslant 1}$ now follows 
from Theorem~\ref{cor:tightness_stationary_sequences} and the combination 
of \eqref{eq:checking_tightness_criterion_martingales}, 
\eqref{eq:limit_of_ANJ} and \eqref{eq:limit_of_BNJ}.
Acounting Lemma \ref{remark_fdd_1} this concludes the 
proof of Theorem~\ref{thm:WIP_Besov_martingales_1}.


\subsection{Proof of Theorem~\ref{thm:WIP_Besov_iid_1}}
  Sufficiency of the condition is contained in 
  Theorem~\ref{thm:WIP_Besov_martingales_1}. 
  {Indeed, we represent the sequence $\prt{Y_j}_{j\geqslant 0}$ 
  by $\prt{f\circ \TT^j}_{j\in \Z}$, that is, $\prt{f\circ \TT^j}_{j\in \Z}$ is 
  an \iid\  sequence and  $Y_j$ has the 
  same distribution as $f\circ \TT^j$. To this aim we define $\Omega=\R^\Z, \cF=\cB^\Z$ 
  and $\PP=\PP_Y^\Z$, where $\PP_Y$ is the distribution of $Y_0$. 
  Let $f((\omega_j))=\omega_0$ for $(\omega_j)\in \R^Z$ and let $\TT:\Omega\to \Omega$ be 
  the shift operator: $\TT((\omega_j))=(\omega_{j+1})$. Next let 
  $\mathcal F_0:=\sigma\prt{f\circ \TT^j, j\leqslant 0 }$. Then $\TT\mathcal F_0
  \subset\mathcal F_0$ and $\E{f\mid \TT\mathcal F_0}
  =0$, since $f$ is independent of $\TT\mathcal F_0$ and centered. Moreover, 
  $\E{f^2\mid \TT\mathcal F_0}=\E{f^2}$, again by independence. Therefore, 
  condition (ii) of \eqref{thm:WIP_Besov_martingales_1} is satisfied. 
  Since $\mathcal I$ is trivial, $\E{f^2\mid \mathcal I}=\E{f^2}$, which gives 
  the convergence \eqref{eq:WIP_Besov_iid_1}.
  }
    
  Let us prove the necessity of \eqref{eq:tail_condition_iid} for the 
  invariance principle in $B_{p,\alpha}^o$. Since the space $B_{p,\alpha}^o$ 
  is a separable Banach space, the sequence $\prt{W_{2^n}}_{n\geqslant 1}$ is 
  tight in $B_{p,\alpha}^o$. Using  Theorem~1 in \cite{MR2483397}, we can find 
  for any positive $\eta$ a number $J_0$ such that 
  \begin{equation*}
   \limsup_{n\to\infty} 
\PP\Big(\sup_{j\geqslant J_0}2^{j\alpha-j/p }\Big(\sum_{r\in D_j}
\abs{W_{2^n}(r^+) -W_{2^n}(r)}^p\Big)^{1/p}>\e\Big)\leqslant \eta.
  \end{equation*}
Therefore, if $n$ is large enough, we have 
\begin{equation*}
\PP\Big(2^{n\alpha-n/p }\Big(\sum_{r\in D_n}
\abs{W_{2^n}\prt{r^+} -W_{2^n}(r)}^p\Big)^{1/p}>\e\Big)  \leqslant 2\eta.
\end{equation*}
Since 
\begin{equation*}
 \sum_{r\in D_n}
\abs{W_{2^n}(r^+) -W_{2^n}(r)}^p=
2^{-np/2}\sum_{l=1 }^{2^{n-1}} \abs{ S_{2l}-S_{2l-1}}^p
=2^{-np/2}\sum_{l=1 }^{2^{n-1}} \abs{X_{2l-1}}^p,
\end{equation*}
we have the convergence in probability of the sequence 
$\prt{2^{n\alpha-j/p } 2^{-np/2}\sum_{l=1 }^{2^{n-1}} \abs{X_{2l-1}}^p  }$ 
to $0$.
By \cite{MR2091561}, this implies that $ 2^{n\alpha-j/p } 2^{-np/2}\PP\ens{\abs{X_1}>2^n}\to 0$, hence 
\eqref{eq:tail_condition_iid} holds.
This ends the proof of Theorem~\ref{thm:WIP_Besov_iid_1}.

\subsection{Proof of Theorem  \ref{thm:counter_example_martingales}}
 
 We first start by a lemma which guarantees the lake of tightness of the 
 partial sum process.
 \begin{Lemma}\label{lem:non_tightness}
  Let $1/p<\alpha<1/2$ and let $f$ be a function such that there exist
  increasing sequences of real numbers $(n_l)_{l\geq 1}$ and 
  $(k_l)_{l\geq 1}$ satisfying the following properties: $k_l/n_l\to 0$ 
  as $l$ goes to infinity and 
  \begin{equation*}
   \inf_{l\geqslant 1}\PP \Big(\frac 1{n_l^{q(p,\alpha)}}\max_{
   1\leqslant k\leqslant k_l  } \frac 1{k^\alpha} \Big(\sum_{i=0}^{n_l-k}\left|
   S_{f,i+k}-S_{f,i}\right|^p\Big)^{1/p}  >1\Big)>0,
  \end{equation*}
  where $q(p,\alpha)$ is given by \eqref{q_alpha_p}.
  Then the sequence $(W_n(f))_{n\geqslant 1}$ is not tight in $B_{p,\alpha}^o$.
 \end{Lemma}
  
  \begin{proof}
  If the sequence $\prt{W_n(f)}_{n\geqslant 1}$ was tight in $B_{p,\alpha}^o$, then 
  we would be able to extract a weakly convergence subsequence of 
  $\prt{W_{n_l}(f)}_{l\geqslant 1}$. Therefore, we can assume without loss of 
  generality  that $\prt{W_{n_l}(f)}_{l\geqslant 1}$ converges in distribution 
  in $B_{p,\alpha}$. Consequently, the sequence 
  $\prt{\sup_{\abs t\leqslant k_l/n_l}   t^{-\alpha}\omega_p\prt{W_{n_l},t}}_{l\geqslant 1}$ 
  should convergence to $0$ in probability as $l$ goes to infinity. But 
  \begin{equation*}
   \sup_{\abs t\leqslant k_l/n_l}   t^{-\alpha}\omega_p\prt{W_{n_l},t} 
   \geqslant 
   \frac{c(p)}{n_l^{q(p,\alpha)}}\max_{
   1\leqslant k\leqslant k_l  } \frac 1{k^\alpha} \Big(\sum_{i=0}^{n_l-k}\left|
   S_{f,i+k}-S_{f,i}\right|^p\Big)^{1/p}  
  \end{equation*}
  for some constant depending only on $p$ (this can be seen by 
  restricting the supremum over the $t$ of the form $k/n_l$ where 
  $1\leqslant k\leqslant k_l$).
  
  \end{proof}

 Let us recall the statement of Lemma~3.8 in \cite{MR1856684}.
 
 \begin{Lemma}\label{lem:positive_entropy}
  Let $(\Omega,\mathcal F,\PP,\TT)$ be an ergodic probability measure preserving system of 
  positive entropy. There exists two $\TT$-invariant sub-$\sigma$-algebras $\mathcal B$ and 
  $\mathcal C$ of $\mathcal F$ and a function $g\colon\Omega\to \R$ such that:
  \begin{itemize}
   \item the $\sigma$-algebras $\mathcal B$ and $\mathcal C$ are independent;
   \item the function $g$ is $\mathcal B$-measurable, takes the values $-1$, $0$ 
   and $1$, has zero mean and the process $(g\circ \TT^n)_{n\in\Z}$ is independent;
   \item the dynamical system $(\Omega,\mathcal C,\PP,\TT)$ is aperiodic.
  \end{itemize}
 \end{Lemma}
 
 In the sequel, we shall assume for simplicity that $\PP\ens{g=1}
 =\PP\ens{g=-1}=1/2$.

 The construction follows the lines of that of Theorem~2.1 in \cite{MR3426520}.
 We define three increasing sequences of positive integers $\left(I_l\right)_{l\geqslant 1}$, 
 $\left(J_l\right)_{l\geqslant 1}$, $\left(N_l\right)_{l\geqslant 1}$ and 
 a sequence of real numbers $\left(L_l\right)_{l\geqslant 1}$ such that 
 $$ \sum_{l=1}^{\infty}\frac 1{L_l}<\infty \label{lac0} \mbox{ and  }
$$ 
 $L_l$ is a continuity point of the cumulative distribution function of 
 the random variable $2^{-1}Y_{1/2,1}$, which is defined in
 \eqref{eq:example_of_convergence_of_functionals}. Now, we define a sequence of real numbers 
  $\left(J_l\right)_{l\geqslant 1}$ in such a way that for each $l\geqslant 1$, 
  \begin{equation}\label{eq:constrainJl}
  \abs{J_l\PP\prt{ \ens{ Y_{1/2,1}>2L_l } } -7/8}\leqslant 1/16.
  \end{equation}
 Now, by Proposition~\ref{prop:example_of_functional}, we can choose for each $l\geqslant 1$ 
 an integer 
 $I_l$ such that 
 \begin{equation}\label{eq:choice_of_Il}
\forall n\geqslant I_l, \quad \abs{\PP\prt{\ens{Y_{n,1/2,1}(g) >2L_l}}
-\PP\prt{\ens{Y_{1/2,1} >2L_l}} } \leqslant \frac 1{lJ_l}.
 \end{equation}
 Let 
 $K_l:=2^{I_l+J_l}$.
 We define the sequence $\prt{N_l}_{l\geqslant 1}$ in such a way that for 
 each $l\geqslant 1$,
 \begin{equation}\label{eq:first_lacunarity_condition_Nl}
 \frac 1{N_l^{1/q(p,\alpha)}}K_l^{1-\alpha}
 \sum_{u=1}^{l-1}\left(\frac{N_u}{2^{I_u}}\right)^{1/q(p,\alpha)}\leqslant 1\mbox{ and }
 \end{equation}
 \begin{equation}\label{eq:second_lacunarity_condition_Nl}
  N_l\sum_{u=l+1}^{+\infty}K_u/N_u<\frac 1{16}.
 \end{equation}

 Using Rokhlin's lemma, we can find for any integer $l\geqslant 1$ a measurable set $
 C_l\in\mathcal C$ such that the sets $\TT^{i}C_l$, $i=0,\dots,N_l-1$ are 
 pairwise disjoint and $\PP\left(\bigcup_{i=0}^{N_l-1}\TT^{i}C_l\right)>1/2$.
 We define for $l\geq 1$ 
 \begin{equation}\label{eq:defintion_of_f_l_counter_example}
  f_l:=\frac 1{L_l}\sum_{j=1}^{J_l}\left(\frac{N_l}{2^{I_l+j}}\right)^{1/q(p,\alpha)}
  \mathbf 1\Big(\bigcup_{i=2^{J_l+j}+1}^{2^{J_l+j+1}}\TT^{N_l-i}C_l\Big)\mbox{ and }
 \end{equation}
 \begin{equation}\label{eq:definition_of_m_counter_example}
  f:=\sum_{l=1}^{+\infty}f_l, \quad m:=g\cdot f.
 \end{equation}
 Note that $\PP\prt{f_l\neq 0}\leqslant K_l/N_l$, hence by 
 \eqref{eq:second_lacunarity_condition_Nl} and the Borel-Cantelli lemma, 
 the function $f$ is well defined almost everywhere.
 Define 
 \begin{equation*}
  \mathcal F_0:=\sigma\left(g\circ \TT^i,i\leqslant 0\right)\vee 
  \mathcal C.
 \end{equation*}

 \begin{Proposition} 
  The $\sigma$-algebra $\mathcal F_0$ satisfies $\TT\mathcal F_0\subset\mathcal F_0$. 
  The function $m$ defined by \eqref{eq:defintion_of_f_l_counter_example} and 
  \eqref{eq:definition_of_m_counter_example} is 
  $\mathcal F_0$-measurable and 
  satisfies $\EE\left[m\mid \TT\mathcal F_0\right]=0$ and 
  $\lim_{t \to  + \infty}t^{q(p,\alpha)}\PP\left\{  
  \left|m\right|>t\right\}=0$.
 \end{Proposition}
 A proof can be found in \cite{MR3426520}
 
 It remains to prove that the sequence $(W_n(m))_{n\geq 1}$ is not tight in 
 $B_{p,\alpha}^o$. 
 
 To this aim, we shall check the conditions of Lemma~\ref{lem:non_tightness}. 
 We first show the following intermediate step.
 
 \begin{Lemma}\label{lem:bound_for_gf_l}
  For each integer $l\geqslant 1$, 
  $$ \PP\Big(\frac 1{N_l^{1/q(p,\alpha)}}\max_{1\leqslant k\leqslant K_l}k^{-\alpha}
 \Big(\sum_{i=0}^{N_l-k}
 \abs{S_{gf_l,i+k} -S_{gf_l,i}}^p\Big)^{1/p} >2\Big)>
 \frac 18.
  $$
 \end{Lemma}

 Let $l\geqslant 1$ be fixed. Assume that $\omega$ belongs to $\TT^{N_l-i_0}$ for some 
 $i_0 \in \left\{K_l,\dots,N_l-1\right\}$. Let $i$ be such that   
 $i_0-2^{I_l+j+1} \leqslant i\leqslant i_0-2^{I_l+j}+1$ for some $j\in 
 \left\{1, \dots,J_l\right\}$. We have 
 \begin{equation*}
  f_l \circ \TT^i(\omega)=\frac 1{L_l}\left(\frac{N_l}{2^{I_l+j}}\right)^{1/q(p,\alpha)}.
 \end{equation*}
 Consequently, for any $k$ such that $2^{I_l+j-1}<k\leqslant 2^{I_l+j}$ 
 and each $i$ such that $i_0-2^{I_l+j+1} \leqslant i\leqslant i_0-k-2^{I_l+j}+1$, we have 
 \begin{equation*}
  \left|S_{gf_l,i+k}-S_{gf_l,i}\right|=
  \frac 1{L_l}\left(\frac{N_l}{2^{I_l+j}}\right)^{1/q(p,\alpha)}
  \left|S_{g,i+k}-S_{g,i}\right|.
 \end{equation*}
It thus follows that 
\begin{multline*}
 \frac 1{N_l^{1/q(p,\alpha)}}\max_{1\leqslant k\leqslant K_l}k^{-\alpha}
 \Big(\sum_{i=0}^{N-k}
 \abs{S_{gf_l,i+k} -S_{gf_l,i}}^p\Big)^{1/p}
 \mathbf 1\left(\TT^{N_l-i_0}C_l\right)\\
 \geqslant  \max_{1\leqslant j\leqslant J_l}
 \max_{2^{I_l+j-1}<k\leqslant 2^{I_l+j}}k^{-\alpha}
  \frac 1{L_l}\left(\frac{1}{2^{I_l+j}}\right)^{1/q(p,\alpha)}
 \Big(\sum_{i=i_0-2^{I_l+j+1}}^{i_0-k-2^{I_l+j}+1}
 \abs{S_{g,i+k} -S_{g,i}}^p\Big)^{1/p}
 \mathbf 1\left(\TT^{N_l-i_0}C_l\right),
\end{multline*}
and using disjointness of the sets $\TT^{N_l-i_0}C_l$, $K_l\leqslant i_0\leqslant N_l-1$, 
we infer that 

\begin{multline}\label{eq:lower_bound_Pl}
 \PP{\Big(\frac 1{N_l^{1/q(p,\alpha)}}\max_{1\leqslant k\leqslant K_l}k^{-\alpha}
 \Big(\sum_{i=0}^{N-k}
 \abs{S_{gf_l,i+k} -S_{gf_l,i}}^p\Big)^{1/p}>2}\Big) 
 \\ 
 \geqslant \PP{\Big(\frac 1{N_l^{1/q(p,\alpha)}}\max_{1\leqslant k\leqslant K_l}k^{-\alpha}
 \Big(\sum_{i=0}^{N-k}
 \abs{S_{gf_l,i+k} -S_{gf_l,i}}^p\Big)^{1/p}>2}\cap 
 \bigcup_{i_0=K_l}^{N_l-1}\TT^{N_l-i_0}C_l\Big) \\
 =\sum_{i_0=K_l}^{N_l-1}\PP\Big(\Big\{\frac 1{N_l^{1/q(p,\alpha)}}\max_{1\leqslant k\leqslant K_l}k^{-\alpha}
 \Big(\sum_{i=0}^{N-k}
 \abs{S_{gf_l,i+k} -S_{gf_l,i}}^p\Big)^{1/p}>2\Big\}\cap 
 \TT^{N_l-i_0}C_l\Big) \\
 \geqslant \sum_{i_0=K_l}^{N_l-1}\PP(A_{i_0} 
\cap \TT^{N_l-i_0}C_l)
\end{multline}
where
$$
A_{i_0}= 
 \Big\{
 \max_{1\leqslant j\leqslant J_l}
 \max_{2^{I_l+j-1}<k\leqslant 2^{I_l+j}}k^{-\alpha}
  \frac 1{L_l}\left(\frac{1}{2^{I_l+j}}\right)^{1/q(p,\alpha)}
 \Big(\sum_{i=i_0-2^{I_l+j+1}}^{i_0-k-2^{I_l+j}+1}
 \abs{S_{g,i+k} -S_{g,i}}^p\Big)^{1/p}>2 \Big\}.
 $$
Since $\TT$ is measure-preserving, the events 
$ A_{i_0}\cap 
 \TT^{N_l-i_0}C_l,\quad K_l\leqslant  i_0 \leqslant N_l-1$
have the same probability, which is equal to 
$\PP(A_{K_l}\cap  \TT^{N_l-K_l}C_l)$.
The events $A_{K_l}$ and 
%
 $\TT^{N_l-K_l}C_l$ belong respectively to $\mathcal B$ and $\mathcal C$, hence 
 they are independent. In view of \eqref{eq:lower_bound_Pl}, we obtain 
 \begin{multline}\label{eq:lower_bound_Pl2}
 \PP{\Big(\frac 1{N_l^{1/q(p,\alpha)}}\max_{1\leqslant k\leqslant K_l}k^{-\alpha}
 \Big(\sum_{i=0}^{N_l-k}
 \abs{S_{gf_l,i+k} -S_{gf_l,i}}^p\Big)^{1/p}>2\Big)} 
 \geqslant \prt{N_l-K_l}\PP(A_{K_l})\PP(C_l)\\
 \geqslant  \PP(A_{K_l})/2.
%
 \end{multline}
 Now, in order to control the latter term, we shall use the following lemma:
 \begin{Lemma}\label{lem:Bonferroni}
  Let $\prt{H_l}_{l\geqslant 1}$ be an increasing sequence of integers. 
  Assume that for each $l\geqslant 1$, the family of events 
  $\prt{A_{l,j}}_{1\leqslant j\leqslant H_l}$ is independent and that 
  $\sum_{j=1}^{H_l}\PP\prt{A_{l,j}}\in [3/4,1]$. Then for each $l\geqslant 1$, 
  \begin{equation*}
   \PP\Big(\bigcup_{j=1}^{H_l} A_{l,j} \Big) \geqslant 1/4.
  \end{equation*}

 \end{Lemma}
  \begin{proof}[Proof of Lemma~\ref{lem:Bonferroni}]
   By Bonferroni's inequality, we have for any $l\geqslant 1$, 
   \begin{equation*}
    \PP\Big(\bigcup_{j=1}^{H_l} A_{l,j} \Big) \geqslant
    \sum_{j=1}^{H_l}\PP\prt{ A_{l,j}}-
    \sum_{1\leqslant i<j\leqslant H_l}\PP\prt{A_{l,i}\cap 
    A_{l,j}}.
   \end{equation*}
  Using independence of $\prt{A_{l,j}}_{1\leqslant j\leqslant H_l}$, we derive that 
  \begin{align*}
    \PP\prt{\bigcup_{j=1}^{H_l} A_{l,j} } &\geqslant
    \sum_{j=1}^{H_l}\PP\prt{ A_{l,j}}-\frac 12\left(2
    \sum_{1\leqslant i<j\leqslant H_l}\PP\prt{A_{l,i}}
    \PP\prt{A_{l,j}}\right)\\
    &= \sum_{j=1}^{H_l}\PP\prt{ A_{l,j}}-\frac 12\left( \prt{\sum_{j=1}^{H_l}
    \PP\prt{A_{l,j} } }^2- \sum_{j=1}^{H_l}\prt{\PP\prt{A_{l,j}} }^2\right)\\
    &\geqslant \sum_{j=1}^{H_l}\PP\prt{ A_{l,j}}-\frac 12
    \prt{ \sum_{j=1}^{H_l}\PP\prt{ A_{l,j}}}^2 \\
    &\geqslant 3/4-1/2=1/4.
  \end{align*}

   \end{proof}
 We now use Lemma~\ref{lem:Bonferroni} with the choices $H_l=J_l$ and 
 \begin{equation*}
  A_{l,j}:=\left\{
 \max_{2^{I_l+j-1}<k\leqslant 2^{I_l+j}}k^{-\alpha}
  \frac 1{L_l}\left(\frac{1}{2^{I_l+j}}\right)^{1/q(p,\alpha)}
 \left(\sum_{i=K_l-2^{I_l+j+1}}^{K_l-k-2^{I_l+j}+1}
 \abs{S_{g,i+k} -S_{g,i}}^p\right)^{1/p}>2 \right\}.
 \end{equation*}
We indeed have, with the notations of \eqref{eq:example_of_convergence_of_functionals} and 
by \eqref{eq:choice_of_Il}, 
\begin{equation*}
 \abs{\sum_{j=1}^{J_l}\PP\prt{A_{l,j}} -J_l \PP\prt{  
 \ens{Y_{1/2,1}>2L_l}
 } } \leqslant \sum_{j=1}^{J_l}
 \frac 1{lJ_l}=1/l
\end{equation*}
hence by \eqref{eq:constrainJl},
\begin{equation*}
 \abs{\sum_{j=1}^{J_l}\PP\prt{A_{l,j}} -7/8} \leqslant 
 \frac 1{16l}+\abs{J_l \PP\prt{  
 \ens{Y_{1/2,1}>2L_l} }-7/8
 } \leqslant \frac  18.
\end{equation*}
 
 We get, in view of \eqref{eq:lower_bound_Pl2} the lower bound 
 \begin{equation*}
   \PP{\ens{\frac 1{N_l^{1/q(p,\alpha)}}\max_{1\leqslant k\leqslant K_l}k^{-\alpha}
 \Big(\sum_{i=0}^{N_l-k}
 \abs{S_{gf_l,i+k} -S_{gf_l,i}}^p\Big)^{1/p}>2} }
 \geqslant \frac 18.
 \end{equation*}
 This concludes the proof of Lemma~\ref{lem:bound_for_gf_l}.

Now, we prove that for any $l\geqslant 1$, 
\begin{equation}\label{eq:verification_non_tightness}
 \PP{\ens{\frac 1{N_l^{1/q(p,\alpha)}}\max_{1\leqslant k\leqslant K_l}k^{-\alpha}
 \Big(\sum_{i=0}^{N_l-k}
 \abs{S_{m,i+k} -S_{m,i}}^p\Big)^{1/p}>1} }
 \geqslant \frac 1{16}.
\end{equation}
We first prove that 
\begin{equation}\label{eq:bound_for_sum_fi_i_leq_l}
 \frac 1{N_l^{1/q(p,\alpha)}}\max_{1\leqslant k\leqslant K_l}k^{-\alpha}
 \Big(\sum_{i=0}^{N_l-k}
 \abs{S_{f'_l,i+k} -S_{f'_l,i}}^p\Big)^{1/p}\leqslant 1,
\end{equation}
where $f'_l:=\sum_{i=1}^lgf_i$. First note that for 
$1\leqslant k\leqslant K_l$ and 
$0\leqslant i\leqslant N_l-k$, 
\begin{align*}
 \abs{S_{f'_l,i+k} -S_{f'_l,i}}
 &\leqslant\sum_{u=1}^{l-1} \abs{S_{gf_u,i+k} -S_{gf_u,i}}
 \leqslant k\sum_{u=1}^{l-1}\max_{i\leqslant v\leqslant i+k-1}\abs{f_u\circ 
 \TT^v}\\
 &\leqslant k\cdot \sum_{u=1}^{l-1} \max_{0\leqslant v\leqslant N_l}\abs{f_u\circ 
 \TT^v},
\end{align*}
hence 
\begin{equation*}
 \frac 1{N_l^{1/q(p,\alpha)}}\max_{1\leqslant k\leqslant K_l}k^{-\alpha}
 \left(\sum_{i=0}^{N_l-k}
 \abs{S_{f'_l,i+k} -S_{f'_l,i}}^p\right)^{1/p}
 \leqslant 
 \frac 1{N_l^{1/q(p,\alpha)}}K_l^{1-\alpha}\cdot \sum_{u=1}^{l-1}\cdot
 \max_{0\leqslant v\leqslant N_l}\abs{f_u\circ \TT^v}.
\end{equation*}
Now, by definition of $f_u$, for each $\omega\in \Omega$, the following inequality holds: 
$\abs{f_u\prt{\omega}}\leqslant \left(\frac{N_u}{2^{I_u}}\right)^{1/q(p,\alpha)}$. 

Consequently, 
\begin{equation*}
 \frac 1{N_l^{1/q(p,\alpha)}}\max_{1\leqslant k\leqslant K_l}k^{-\alpha}
 \left(\sum_{i=0}^{N_l-k}
 \abs{S_{f'_l,i+k} -S_{f'_l,i}}^p\right)^{1/p}
 \leqslant  \frac 1{N_l^{1/q(p,\alpha)}}K_l^{1-\alpha}
 \sum_{u=1}^{l-1}\left(\frac{N_u}{2^{I_u}}\right)^{1/q(p,\alpha)},
 \end{equation*}
and this term does not exceed $1$ by \eqref{eq:first_lacunarity_condition_Nl}. This 
proves \eqref{eq:bound_for_sum_fi_i_leq_l}

Now, defining $f''_l:=\sum_{u=l+1}^{+\infty}gf_u$, we have 
\begin{multline*}
 \PP{\ens{\frac 1{N_l^{1/q(p,\alpha)}}\max_{1\leqslant k\leqslant K_l}k^{-\alpha}
 \Big(\sum_{i=0}^{N_l-k}
 \abs{S_{f''_l,i+k} -S_{f''_l,i}}^p\Big)^{1/p}\neq 0} }\\
 \leqslant \sum_{u=l+1}^{+\infty}
 \PP{\ens{\max_{1\leqslant k\leqslant K_l}k^{-\alpha}
 \left(\sum_{i=0}^{N_l-k}
 \abs{S_{gf_u,i+k} -S_{gf_u,i}}^p\right)^{1/p}\neq 0} }\\
 \leqslant \sum_{u=l+1}^{+\infty}\PP{\ens{\max_{0\leqslant v\leqslant N_l-1}
 \abs{gf_u\circ \TT^v}\neq 0  }}\\
 \leqslant N_l\sum_{u=l+1}^{+\infty}\PP{\ens{
 \abs{gf_u}\neq 0  }}\leqslant N_l\sum_{u=l+1}^{+\infty}\PP{\ens{
 f_u\neq 0  }}.
\end{multline*}
By constructing of $f_u$, we have $\PP\prt{\ens{
 f_u\neq 0  }}\leqslant K_u/N_u$, hence 
 \begin{equation}\label{eq:bound_for_sum_fi_i_geq_l}
 \PP{\ens{\frac 1{N_l^{1/q(p,\alpha)}}\max_{1\leqslant k\leqslant K_l}k^{-\alpha}
 \left(\sum_{i=0}^{N_l-k}
 \abs{S_{f''_l,i+k} -S_{f''_l,i}}^p\right)^{1/p}\neq 0} }
 \leqslant N_l\sum_{u=l+1}^{+\infty}\frac{K_u}{N_u}\leqslant 1/16,
 \end{equation}
 by \eqref{eq:second_lacunarity_condition_Nl}. 
 
 Thus \eqref{eq:verification_non_tightness} follows from the combination of 
 Lemma~\ref{lem:bound_for_gf_l},
 \eqref{eq:bound_for_sum_fi_i_leq_l} and 
 \eqref{eq:bound_for_sum_fi_i_geq_l}. This ends the proof of 
 Theorem~\ref{thm:counter_example_martingales}.

\section{Proofs: the case $\alpha\leqslant 1/p$}
\label{sec:proof_alpha_leq_1/p}

We start with the following lemma which reduces the proof of convergence to that of tightness.
\begin{Lemma}\label{remark_fdd_2}  Let $p\geqslant 1$ and 
$0\leqslant\alpha\leqslant \min\{1/2, 1/p\}$. Assume that $Z$ is a random 
element in  $B^o_{p, \alpha}$. Then for any stationary sequence $\prt{f\circ \TT^j}$ if 
\begin{itemize}
\item[$(i)$] $W_{f,n}\wc Z$ in $L_p[0, 1]$, and 
\item[$(ii)$] $(W_{f,n})$ is tight in $B^o_{p, \alpha}$, 
\end{itemize}
then $W_{f,n}\wc Z$ in $B^o_{p, \alpha}$. 
\end{Lemma}
\begin{proof} From (ii) we have that each subsequence of $(W_{f,n})$ has further subsequence that converges in distribution. If $W_{f,n'}\wc Y'$ and $W_{f, n''}\wc Y''$ then  we have that for Franklin basis $(f_k)$ it holds 
that $\ip{W_{f,n'}, f_k}\wc \ip{Y', f_k}$ and  
$\ip{W_{f,n''}, f_k}\wc \ip{Y'', f_k}$ for any  $k$. But (i) gives that both $\ip{Y', f_k}$ 
and $\ip{Y'', f_k}$ have the same distribution as $\ip{Z, f_k}$.
Since coefficients $\ip{Z, f_k}$ determines the distribution of $Z$  we can conclude 
that $Y'$ and $Y''$ are equally 
distributed with $Z$. This ends the proof.
\end{proof}

\subsection{Proof of Theorem \ref{thm:WIP_Besov_martingales_2}}  
Due to continuity of the embedding 
$B^o_{2,\alpha}\hookrightarrow B^o_{p, \alpha}$ if $1\leqslant p\leqslant 2$ and $0\leqslant \alpha<1/2$.
it is enough to prove the 
case where either $p=2$ and $0\leqslant\alpha<1/2$ or $p>2$ and $0\leqslant \alpha\leqslant 1/p$. 
%

Recall $W_{f,n}=n^{-1/2}\zeta_{f,n}$. 
We shall prove  for each $\e>0$
\begin{equation}\label{eq:1} 
\lim_{\delta\to 0}\sup_{n\geqslant 1} I_n(\delta, \e)=0,
\end{equation}
where
$$
I_n(\delta, \e)=\PP\Big(\delta^{-\alpha}\sup_{|h|\le\delta}\Big(\int_0^1|W_{f,n}(t+h)-W_{f, n}(t)|^p\dd t\Big)^{1/p}>\e \Big).
$$
Since the function $W_{f,n}(t), 0\leqslant t\leqslant 1$  is affine in each interval 
$\left((k-1)/n, k/n\right]$, it holds  for  $s, t\in [(k-1)/n, k/n],$ 
\begin{equation}\label{eq:2a}
\abs{W_{f,n}(s)-W_{f,n}(t)}\leqslant n^{1/2}|s-t|\cdot \abs{f\circ \TT^k}.
\end{equation} 
This observation leads to  
$$
\omega_p(W_{f,n}, \delta)\leqslant c_p\left[U_{f,n}(\delta)+V_{f,n}(\delta)\right],
$$
where $c_p>0$ is a constant depending on $p$ only, 
\begin{align*}
U_{f,n}(\delta)&:=\min\{\delta, n^{-1}\}n^{1/2}
\prt{\frac{1}{n}\sum_{k=1}^n|f\circ \TT^k|^p}^{1/p},\\
V_{f,n}(\delta)&:=n^{-1/2}\max_{1\leqslant \ell\leqslant \lfloor n\delta\rfloor }
\prt{\frac{1}{n}\sum_{k=0}^{n-\ell}\abs{\sum_{j=k+1}^{k+\ell}f\circ \TT^j}^p}^{1/p}.
\end{align*}
As a consequence in order to establish \eqref{eq:1} we have to prove 
\begin{align}
&\lim_{\delta\to 0}\sup_{n\geqslant 1} \PP\ens{\delta^{-\alpha}U_{f,n}(\delta)\geqslant \e}=0,\label{1task}\\
&\lim_{\delta\to 0}\sup_{n\geqslant 1/\delta} 
\PP\ens{\delta^{-\alpha}V_{f,n}(\delta)> \e}=0.\label{2task}
\end{align}
Consider first \eqref{1task} and start with $p=2$ and $0\leqslant\alpha<1/2$.
By Chebyshev inequality 
\begin{align*}
\PP\ens{\delta^{-\alpha}U_{f,n}(\delta)\geqslant \e}&\leqslant \e^{-2}\delta^{-2\alpha}
\min\{\delta^2, n^{-2}\}\E{\sum_{k=1}^n |f\circ \TT^k|^2}\\
&\leqslant \e^{-2}\delta^{-2\alpha}\min\{\delta^2, n^{-2}\}
n\E{f^2}\leqslant \e^{-2}\delta^{1-2\alpha}
\end{align*}
and \eqref{1task} follows in this case. 
%
%
Now let $p>2$ and $0\leqslant \alpha\leqslant 1/p$. 
For this case  we shall use truncation. Set for $\tau>0$,
$$
f'=f\bm{1}(|f|\le \tau\sqrt{\max\{n, \delta^{-1}\}}),\quad f''=f-f'.
$$
Then $\PP(\delta^{-\alpha}U_{f,n}(\delta)>\e)\le n\PP(|f|\ge \tau\sqrt{\max\{n, \delta^{-1}\}})+\PP(\delta^{-\alpha}U_{f',n}(\delta)>\e)$ and, since 
$$
n\PP(|f|\ge \tau\sqrt{\max\{n, \delta^{-1}\}})\le \tau^{-2}\E f^2\bm{1}(|f|\ge \sqrt{\delta^{-1}})
$$
we reduce the proof of \eqref{1task} to 
\begin{equation}\label{1atask}
\lim_{\delta\to 0}\sup_{n\geqslant 1} \PP\ens{\delta^{-\alpha}U_{f',n}(\delta)\geqslant \e}=0.
\end{equation}
We have by Chebyshev inequality,
\begin{align*}
\PP\ens{\delta^{-\alpha}U_{f',n}(\delta)\geqslant \e}&\leqslant \e^{-p}\delta^{-p\alpha}
\min\{\delta^p, n^{-p}\}n^{p/2}\EE\Big[n^{-1}\sum_{k=1}^n |f'\circ \TT^k|^p\Big]\\
&\leqslant \e^{-p}\delta^{-p\alpha}\min\{\delta^p, n^{-p}\}n^{p/2}
\E{(f')^p}\\
&\leqslant \e^{-p}\delta^{-p\alpha}\min\{\delta^p, n^{-p}\}n^{p/2} \tau^{p-2}(\max\{n, \delta^{-1}\})^{(p-2)/2}\EE f^2\\
&\leqslant \e^{-p}\tau^{p-2}(\min\{\delta, n^{-1}\})^{1-p\alpha}\EE f^2. 
\end{align*}
Hence
$$
\lim_{\delta\to 0}\sup_{n\geqslant 1} \PP\ens{\delta^{-\alpha}U_{f',n}(\delta)\geqslant \e}\le \e^{-p}\tau^{p-2}.
$$
Since $\tau>0$ is arbitrary, the limit is indeed zero, and  the proof of (\ref{1atask}) is completed. 

To prove \eqref{2task} we start again with the case  $p=2$ and $0\leqslant\alpha<1/2$.  In 
this case Chebyshev inequality along  with stationarity and Doob-Kolmogorov 
inequality yields
\begin{align*}
\PP(\delta^{-\alpha}V_{f,n}(\delta)>\e)&\leqslant 
\e^{-2}\delta^{-2\alpha}\EE( V_{f,n}(\delta))^2\leqslant \e^{-2}\delta^{-2\alpha}n^{-1}
\EE{\max_{1\leqslant \ell\leqslant \lfloor n\delta\rfloor }\Big(\sum_{j=1}^\ell 
f\circ \TT^{j}\Big)^2}\\
&\leqslant \e^{-2}\delta^{1-2\alpha}\EE{f^2}
\end{align*} 
and \eqref{2task} follows.
This ends the proof of \eqref{2task} in the case $p=2$.

Assume that $p>2$ and $\alpha\leqslant 1/p$.
Let us fix $\varepsilon>0$. 
Define for any $\delta \in(0,1)$ and any integer $n\geqslant 1/\delta$ the events 
$$ A_{n,\delta}:=\prt{\delta^{-\alpha} n^{-1/2}\max_{1\leqslant \ell\leqslant \lfloor n\delta\rfloor }
\prt{\frac{1}{n}\sum_{k=0}^{n-\ell}\abs{\sum_{j=k}^{k+\ell-1}f\circ \TT^j}^p}^{1/p}>
\varepsilon
 }
$$ 
$$ B_{n,\delta,\tau}:=\prt{\max_{1\leqslant \ell\leqslant \lfloor n\delta\rfloor }
 n^{-1/2}\max_{1\leqslant k\leqslant n-l}\abs{ S_{f,k+l}-S_{f,k} }  \geqslant 
 \e^{\frac p{p-2}} \tau^{\frac{p}{p-2}} }, 
$$ 
where $\tau$ is an arbitrary but fixed positive number. 
 We have the bound
$$ \PP\prt{B_{n,\delta,\tau}}\leqslant 
 \PP\prt{\sup_{   0\leqslant s<t\leqslant 1,\abs{t-s}<\delta } 
 n^{-1/2}\abs{\zeta_{f,n}(t) -\zeta_{f,n}(s)   } \geqslant 
 \e^{\frac p{p-2}} \tau^{\frac{p}{p-2}}  }.
$$ 
Since the sequence $\prt{\zeta_{f,n}}$ is tight in the space $C[0,1]$
(see \cite{Billingsley:1968}), we have 
\begin{equation} \label{eq:intermediate_step_case_alpha_leq_1/p}
 \lim_{\delta\to 0} \sup_{n\geqslant 1/\delta}
  \PP\prt{B_{n,\delta,\tau}}=0.
\end{equation}
Now, note that on $A_{n,j}\cap B_{n,j,\tau}^c$, we have for any 
$0\leqslant \ell\leqslant \lfloor n\delta\rfloor $, 
\begin{align}
\sum_{k=0}^{n-\ell}\abs{\sum_{j=k}^{k+\ell-1}f\circ \TT^j}^p
&\leqslant \sum_{k=0}^{n-\ell}\abs{\sum_{j=k}^{k+\ell-1}f\circ \TT^j}^2
\max_{1\leqslant \ell\leqslant \lfloor n\delta\rfloor }
\prt{ n^{-1/2}\max_{1\leqslant k\leqslant n-l}\abs{ S_{f,k+\ell}-S_{f,k} } }^{p-2}
n^{(p-2)/2}\\
 &\leqslant \sum_{k=0}^{n-\ell}\abs{\sum_{j=k}^{k+\ell-1}f\circ \TT^j}^2
 n^{(p-2)/2}\e^{p} \tau^{p}.
\end{align}
Therefore, we have 
\begin{align*}
 \e&<\delta^{-\alpha} n^{-1/2}\max_{1\leqslant \ell\leqslant \lfloor n\delta\rfloor }
\prt{\frac{1}{n}\sum_{k=0}^{n-\ell}\abs{\sum_{j=k}^{k+\ell-1}f\circ \TT^j}^p}^{1/p} \\
&\leqslant \delta^{-\alpha}n^{-1/2}n^{-1/p}\prt{n^{(p-2)/2}\e^{p} \tau^{p} }^{1/p}
\max_{1\leqslant \ell\leqslant \lfloor n\delta\rfloor }
\prt{\sum_{k=0}^{n-\ell}\abs{\sum_{j=k}^{k+\ell-1}f\circ \TT^j}^2}^{1/p}\\
&=\delta^{-\alpha}n^{-2/p}\e \tau 
\max_{1\leqslant \ell\leqslant \lfloor n\delta\rfloor }
\prt{\sum_{k=0}^{n-\ell}\abs{\sum_{j=k}^{k+\ell-1}f\circ \TT^j}^2}^{1/p},
\end{align*}
and we infer that 
$$ \PP\prt{A_{n,\delta}\cap B_{n,\delta,\tau}^c} 
 \leqslant \PP\prt{\delta^{-\alpha p}n^{-2} \tau^{p}
 \max_{1\leqslant \ell\leqslant \lfloor n\delta\rfloor }
 \sum_{k=0}^{n-\ell}
 \abs{\sum_{j=k}^{k+\ell-1}f\circ \TT^j}^2>1}.
$$ 
By Markov's inequality, stationarity and Doob's inequality, we have 
\begin{align*}
 \PP\prt{A_{n,\delta}\cap B_{n,\delta,\tau}^c} 
 &\leqslant \delta^{-\alpha p}n^{-2} \tau^{p}\E{\max_{1\leqslant \ell\leqslant \lfloor n\delta\rfloor }
 \sum_{k=0}^{n-\ell}
 \abs{\sum_{j=k}^{k+\ell-1}f\circ \TT^j}^2} \\
 &\leqslant \delta^{-\alpha p}n^{-2} \tau^{p}\E{
 \sum_{k=0}^{n}\max_{1\leqslant \ell\leqslant \lfloor n\delta\rfloor }
 \abs{\sum_{j=k}^{k+\ell-1}f\circ \TT^j}^2}\\
 &=\delta^{-\alpha p}n^{-1} \tau^{p}\E{
 \max_{1\leqslant \ell\leqslant \lfloor n\delta\rfloor }
 \abs{S_{f,\ell}}^2}\\
 &\leqslant 2\delta^{1-\alpha p} \tau^{p}\E{f^2}.
\end{align*}
Since $p\alpha\leqslant 1$, we get 
$$ \PP\prt{A_{n,\delta}\cap B_{n,\delta}^c} \leqslant 2\tau^{p}\E{f^2}
$$ 
and since $\tau$ is arbitrary, we get
$$ \lim_{\delta\to 0}\sup_{n\geqslant 1/\delta}
 \PP\prt{A_{n,\delta}} =0
$$ 
in view of \eqref{eq:intermediate_step_case_alpha_leq_1/p}.
This concludes the proof of \eqref{2task} and that of
Theorem~\ref{thm:WIP_Besov_martingales_2}.

\subsection{Proof of Theorem \ref{thm:WIP_Besov_iid_2}}

 It follows from Theorem \ref{thm:WIP_Besov_martingales_2} and the same 
 arguments as used in the proof of Theorem \ref{thm:WIP_Besov_iid_1}.

\section{Some applications}

As already was mentioned in the introduction, a choice of  functional
spaces for polygonal line processes is usually inspired by  possible applications
in statistics via continuous mappings: if $W_{f,n}\wc W$ in the
space $B^o_{p, \alpha}$, then $T(W_{f,n})\to T(W)$ for any continuous 
function $T:B^o_{p, \alpha}\to \R$. This general observation can be used, for example,  to analyse so called $k$-scan processes 
%
%
$$
S^{(i)}_{f,k}=\sum_{j=i}^{i+k-1} f\circ \TT^j,
\quad i=1, \dots, n-k+1.
$$
%

\begin{Proposition}\label{prop:example_of_functional}
 Let $f$ be a function such that the sequence $\left(W_n(f)\right)_{n\geqslant 1}$ converges 
 to a standard Brownian motion $W$ in 
 $B_{p,\alpha}^o$, where $1/p<\alpha<1/2$. For each 
 $a,b\in [0,1]$ such that $a<b$, we define 
$$  Y_{n,a,b}(f):=n^{-1/{q(p,\alpha)}}\max_{
  \lfloor an\rfloor< k\leqslant \lfloor bn\rfloor+1}\frac 1{k^\alpha}
  \Big(\sum_{i=0}^{n-k}\abs{S^{(i)}_{f,k}}^p\Big)^{1/p}.
$$ 
 Then the following convergence holds:
 \begin{equation}\label{eq:example_of_convergence_of_functionals}
   Y_{n,a,b}(f)\wc  Y_{a,b}:=
  \sup_{a< t\leqslant b}t^{-\alpha}
  \Big(\int_{I_t}\abs{W(s+t)-W(s)}^p\mathrm ds\Big)^{1/p}.
 \end{equation}
\end{Proposition}
\begin{proof}
%
  Let use define a functional $F\colon B_{p,\alpha}^o\to \R$ by 
  \begin{equation*}
  F(x):= \sup_{a< t\leqslant b}t^{-\alpha}
  \left(\int_{I_t}\abs{x(s+t)-x(s)}^p\mathrm ds\right)^{1/p}.
  \end{equation*}
  Then $F$ is continuous with respect to the topology of 
  $B_{p,\alpha}^o$ and $F(W)=Y_{a, b}$.  We thus have $F\prt{W_n\prt f} \wc 
  Y_{a, b}$. To conclude that \eqref{eq:example_of_convergence_of_functionals} holds, it 
  suffices to prove that 
  \begin{equation*}
  Z_n:=F\prt{W_n\prt f}- Y_{n,a,b}(f)\to 0\mbox{ in probability as }n\to +\infty.
  \end{equation*}
 First note that $Z_n$ is non-negative. Second, we have 
 \begin{align*}
 F\prt{W_n\prt f}&\leqslant 
 \max_{\lfloor an\rfloor \leqslant  k\leqslant \lfloor bn\rfloor +1}\max_{k/n\leqslant t<(k+1)/n}
 t^{-\alpha}
  \left(\int_{I_t}\abs{W_n(f,s+t)-W_n(f,s)}^p\mathrm ds\right)^{1/p}\\
  &\leqslant 
 \max_{\lfloor an\rfloor \leqslant  k\leqslant \lfloor bn\rfloor +1}\max_{k/n\leqslant t<(k+1)/n}
 \left(\frac kn\right)^{-\alpha}
  \left(\int_{I_t}\abs{W_n(f,s+t)-W_n(f,s)}^p\mathrm ds\right)^{1/p}.
 \end{align*}
Let $k$ be an integer such that $\lfloor an\rfloor \leqslant  k\leqslant \lfloor bn\rfloor +1$ 
and let $t$ be a real number such that $k/n\leqslant t<(k+1)/n$. Then 
\begin{multline*}
 \abs{\left(\int_{I_t}\abs{W_n(s+t)-W_n(s)}^p\mathrm ds\right)^{1/p}-
 \left(\int_{I_{k/n}} \abs{W_n(s+k/n)-W_n(s)}^p\mathrm ds\right)^{1/p}}\\
 \leqslant 
 \left(\int_{1-t}^{1-k/n}\abs{W_n(f,s+k/n)-W_n(s)}^p\mathrm ds\right)^{1/p}
 +\left(\int_{0}^{1-t}\abs{W_n(f,s+t)-W_n(s+k/n)}^p\mathrm ds\right)^{1/p}
 \\
 \leqslant \left(\int_{1-(k+1)/n}^{1-k/n}\abs{W_n(f,s+k/n)-W_n(s)}^p\mathrm ds\right)^{1/p}
 + \omega_p\prt{W_n(f),t-k/n }  \\
=\left(\int_{1-(k+1)/n}^{1-k/n}\abs{W_n(f,s+k/n)-W_n(s)}^p\mathrm ds\right)^{1/p}
 + \omega_p\prt{W_n(f),t-k/n },
\end{multline*}
which implies that 
\begin{multline}\label{eq:intermediate_step_example_of_convergence}
 F\prt{W_n\prt f}-Y_{n,a,b}(f)
 \leqslant  \max_{\lfloor an\rfloor \leqslant  k\leqslant \lfloor bn\rfloor +1}
 \left(\frac kn\right)^{-\alpha}
 \left(\int_{1-(k+1)/n}^{1-k/n}\abs{W_n(f,s+k/n)-W_n(s)}^p\mathrm ds\right)^{1/p}\\
 +
  \max_{\lfloor an\rfloor \leqslant  k\leqslant \lfloor bn\rfloor +1}
  \left(\frac 1k\right)^{\alpha}\sup_{\abs \delta\leqslant 1/n}
  \delta^{-\alpha}\omega_p\prt{W_n(f),\delta }\\
  \leqslant  \max_{1\leqslant  k\leqslant n-1}\left(\frac nk\right)^{\alpha}
 \left(\int_{1-(k+1)/n}^{1-k/n}\abs{W_n(f,s+k/n)-W_n(s)}^p\mathrm ds\right)^{1/p}\\
 +\sup_{0<\abs \delta\leqslant 1/n}
  \delta^{-\alpha}\omega_p\prt{W_n(f),\delta }.
\end{multline}
By Theorem~\ref{thm:necessary_and_sufficient_condition_for_tightness}, the second 
term in the right hand side of \eqref{eq:intermediate_step_example_of_convergence} goes
to zero in probability. Therefore, it suffices to prove that 
\begin{equation}\label{eq:intermediate_step_example_of_convergence_goal}
 \max_{1\leqslant  k\leqslant n-1}\left(\frac nk\right)^{\alpha}
 \left(\int_{1-\frac{k+1}n}^{1-\frac kn}\abs{W_n\prt{f,s+\frac kn}-W_n(f,s)}^p
 \mathrm ds\right)^{\frac 1p} \to 
 0\mbox{ in probability as }n\to +\infty.
\end{equation}
To see this, we start from the inequality 
\begin{multline*}
 \max_{1\leqslant  k\leqslant n-1}\left(\frac nk\right)^{\alpha}
 \left(\int_{1-(k+1)/n}^{1-k/n}\abs{W_n\prt{f,s+\frac kn}-W_n(f,s)}^p\mathrm ds\right)^{1/p} \\
 \leqslant \sup_{0<t<1}t^{-\alpha}
 \left(\int_{1-t-1/n}^{1-t}\abs{W_n(f,s+t)-W_n(f,s)}^p\mathrm ds\right)^{1/p}\\
 =\sup_{0<t<1}t^{-\alpha}
 \left(\int_{1-1/n}^{1}\abs{W_n(f,s)-W_n(f,s-t)}^p\mathrm ds\right)^{1/p}.
\end{multline*}
Let $N$ be a fixed integer. The functional 
\begin{equation*}
 G_N\colon B_{p,\alpha}^o\to \R, \quad 
 G_N(x)=\sup_{0<t<1}t^{-\alpha}
 \left(\int_{[1-1/N,1]\cap I_{-t}}\abs{x(s)-x(s-t)}^p\mathrm ds\right)^{1/p}
\end{equation*}
is continuous. Therefore, if $\varepsilon>0$ is a continuity point of 
the cumulative distribution function of $G_N(W)$ for each $N$, 
we have
\begin{multline}\label{eq:intermediate_step_example_of_convergence2}
\limsup_{n\to +\infty} \PP  
 \Big(\max_{1\leqslant  k\leqslant n-1}\left(\frac nk\right)^{\alpha}
 \left(\int_{1-(k+1)/n}^{1-k/n}\abs{W_n(f,s+k/n)-W_n(f,s)}^p\mathrm ds\right)^{1/p}>\varepsilon 
 \Big)\\
 \leqslant 
 \PP\Big( \sup_{0<t<1}t^{-\alpha}
 \Big(\int_{[1-1/N,1]\cap I_{-t}}\abs{W(s)-W(s-t)}^p\mathrm ds\Big)^{1/p}
 >\varepsilon \Big).
\end{multline}
Now, take $q>p$. Using H\"older's inequality 
with the exponents $q/p$ and $q/(q-p)$, we have 
\begin{multline*}
 \sup_{0<t<1}t^{-\alpha}
 \left(\int_{[1-1/N,1]\cap I_{-t}}\abs{W(s)-W(s-t)}^p\mathrm ds\right)^{1/p}\\
 \leqslant \sup_{0<t<1}t^{-\alpha}
 \left(\int_{I_{-t}}\abs{W(s)-W(s-t)}^q\mathrm ds\right)^{1/q}N^{-(p-q)/q} 
 \leqslant \lVert W\rVert_{q,\alpha}N^{-(p-q)/q}.
\end{multline*}
Since $\lVert W\rVert_{q,\alpha}$ is almost surely finite, we get 
from \eqref{eq:intermediate_step_example_of_convergence2} that 
\eqref{eq:intermediate_step_example_of_convergence_goal} holds.
\end{proof}
 
Statistics based on $k$-scan processes can be used to detect epidemic change in the 
mean of a sample of size $n$ (see, e.g., \cite{R-S:2004}, and reference therein).
More precisely, given a sample $X_1, X_2, \dots, X_n$, consider the model 
$$
X_i=\mu\bm{1}_{(k^*, m^*]}(i)+ \e_i, \quad i=1, \dots, n,
$$
where $(\e_i)$ is a stationary  sequence, $\mu\not=0$ and $k^*, m^*$ are unknown parameters of the model. We want to test the null hypothesis $H_0:\ \ \mu=0$ against the alternative $\mu\not=0$. To this aim one can consider the statistics 
$$
T_{n}=n^{-1/{q(p,\alpha)}}\max_{0< k\leqslant n}\frac 1{k^\alpha}   \Big(\sum_{i=0}^{n-k}\abs{X_i+\cdots +X_{i+k}}^p\Big)^{1/p}.
$$
Under null its limit is defined by Proposition \ref{prop:example_of_functional} provided $(\e_i)$ satisfies the weak invariance principle in the Besov space $B^o_{p, \alpha}$.  Under alternative then we see that
$$
T_n=\Big(1-\frac{h^*}{n}\Big)^{1/p}\Big(\frac{h^*}{n}\Big)^{1-\alpha}n^{1/2}+O_P(1)
$$
as $n\to\infty$, where $h^*=m^*-k^*$ is the duration  of the epidemic state. 
\def\polhk\#1{\setbox0=\hbox{\#1}{{\o}oalign{\hidewidth
  \lower1.5ex\hbox{`}\hidewidth\crcr\unhbox0}}}


\end{document}